%% file: main.tex
\documentclass[12pt,a4paper]{amsart}
\usepackage{amssymb}
\usepackage{amsmath}
\usepackage{xcolor}

\usepackage{hyperref}
\usepackage[capitalize]{cleveref}
\usepackage{autonum}

\usepackage{enumitem}

\usepackage{geometry}

\title[Local Invariants of H-type foliations]{Local Invariants and Geometry of \\ the sub-Laplacian on H-type foliations}

\address{$^{1,3}$Institut für Analysis, Leibniz Universität, Welfengarten 1, 30167 Hannover, Germany}
\address{$^{2,4}$Matematisk Institutt, Universitetet i Bergen, Allegaten 41, 5007 Bergen, Norway}

\author[W. Bauer]{Wolfram Bauer$^1$}
\email{$^1$bauer@math.uni-hannover.de}

\author[I. Markina]{Irina Markina$^2$}
\email{$^2$irina.markina@uib.no}
\thanks{The work of the second author was partially supported by the project Pure Mathematics in Norway, funded by Trond Mohn Foundation and Troms\o Research Foundation.}

\author[A. Laaroussi]{Abdellah Laaroussi$^3$}
\email{$^3$abdellah.laaroussi@math.uni-hannover.de}

\author[G. Vega-Molino]{Gianmarco Vega-Molino$^4$}
\email{$^4$gianmarco.vega-molino@uib.no}
\thanks{The fourth author is partially supported by the Trond Mohn Foundation - Grant TMS2021STG02 (GeoProCo).}

\input{macros.tex}

\begin{document}

\begin{abstract}
$H$-type foliations $(\M,\Ho,g_\Ho)$ are studied in the framework of sub-Riemannian geometry with bracket generating distribution defined as the bundle transversal to the fibers. Equipping $\M$ with the Bott connection we consider the scalar horizontal curvature $\kh$ as well as a new local invariant $\tv$ induced from the vertical distribution. We extend recent results on the small-time asymptotics of the sub-Riemannanian heat kernel on quaternion-contact (qc-)manifolds due to A. Laaroussi and we express the second heat invariant in sub-Riemannian geometry as a linear combination of $\kh$ and $\tv$. The use of an analog to normal coordinates in Riemannian geometry that are well-adapted to the geometric structure of $H$-type foliations allows us to consider the pull-back of Kor\'{a}nyi balls to $\M$.  We explicitly obtain the first three terms in the asymptotic expansion of their Popp volume for small radii. Finally, we address the question of when $\M$ is locally isometric as a sub-Riemannian manifold to its $H$-type tangent group. 
\end{abstract}

\maketitle
\setcounter{tocdepth}{1} %only sections, not subsections
\tableofcontents
\thispagestyle{empty}
 
%%%%%%%%%%%%%%%%%%%%%%%%%%%%%%%%%%%%%%%%%%%%%%%%%%%%%%% 
 
\input{sections/sec1_introduction}
\input{sections/sec2_H-type-foliations}
\input{sections/sec3_privileged-coordinates}
\input{sections/sec4_popp}
\input{sections/sec5_heat-invariants}
\input{sections/sec6_open-problems}
\input{sections/sec7_appendix}

%%%%%%%%%%%%%%%%%%%%%%%%%%%%%%%%%%%%%%%%%%%%%%%%%%%%%%%

\newpage
\bibliographystyle{abbrv}
\bibliography{biblio}

\end{document}

%% file: macros.tex
\newcommand{\R}{\mathbb{R}}
\newcommand{\M}{\mathbb{M}}
\newcommand{\s}[1]{\Gamma(#1)}
\newcommand{\ve}{\varepsilon}
\newcommand{\G}{\mathrm{G}}

\newcommand{\Ho}{\mathcal{H}}
\newcommand{\Vr}{\mathcal{V}}
\newcommand{\kh}{\kappa_\Ho}
\newcommand{\kv}{\kappa_\Vr}
\newcommand{\tv}{\tau_\Vr}
\newcommand{\rh}{n}
\newcommand{\rv}{m}
\newcommand{\ho}{{\hat\omega}}
\newcommand{\hor}{{\ho_r}}
\newcommand{\DsR}{\Delta_\mathrm{sub}}
\newcommand{\hDsR}{\hat\Delta_\mathrm{sub}}

\newcommand{\tr}{\mathrm{tr}}
\newcommand{\End}{\mathrm{End}}

\newcommand{\Id}{\mathrm{Id}}
\newcommand{\Lie}{\mathcal{L}}
\newcommand{\ord}{\mathrm{ord}}
\newcommand{\emphs}[1]{{\it{#1}}}
\newcommand{\defemphs}[1]{{#1}}

\newtheorem{thm}{Theorem}[section]

\newtheorem{lem}[thm]{Lemma}
\newtheorem{prop}[thm]{Proposition}
\newtheorem{notation}[thm]{Notation}
\newtheorem{rem}[thm]{Remark}

\newtheorem{defn}[thm]{Definition}

\newtheorem{inthm}{Theorem}

\numberwithin{equation}{section}

%% file: sections/sec1_introduction.tex
%%%%%%%%%%%%%%%%%%%%%%%%%%%%%%%%%%%%%%%%%%%%%%%%

\section{Introduction}
Sub-Riemannian geometry and the analysis of intrinsically associated differential equations, such as the sub-Riemannian heat or sub-Riemannian wave equation, have attracted increasing interest during the last decades \cite{ABB20,ABGR09,BBN19,BG17,BL22,BGG96,BGS84,B89,VHT21,L21,SM12}. Sub-Riemannian geometry is the study of bracket generating metric distributions inside the tangent bundle of a given manifold. On the one hand, such geometries allow for the definition of a metric space structure, of tangent groups or geodesic curves. On the other hand, on a regular sub-Riemannian manifold an analytic object, namely the intrinsic hypoelliptic sub-Laplacian, can be defined in form of a second order differential operator \cite{ABGR09,BR13} and generalizes the Beltrami-Laplace operator in Riemannian geometry. The sub-Riemannian heat kernel $K_{\mathrm{sub}}$ is the fundamental solution of the corresponding heat operator and it is expected to encode specific geometric data of the sub-Riemannian structure. With varying degrees of generality the existence of the small time asymptotic expansion of $K_{\mathrm{sub}}$ or of the sub-Riemannian heat trace on compact manifolds has been obtained by analytical or stochastic methods \cite{BGG96,B89,DH20,VHT21,IT20,SM12}.  

In some cases even explicit formulas for the sub-Riemannian heat kernel are available \cite{BB08, BC21, BW13, BW14, BGG96} and are of great use in the asymptotic analysis. However, due to a large variety of sub-Riemannian structures (e.g. see \cite{M02} for many examples) it seems that a geometric interpretation of the coefficients (heat invariants) in the small time asymptotic expansion of $K_{\mathrm{sub}}$ requires a case-by-case analysis. We mention the results for unimodular Lie groups \cite{SM12}, particular sub-Riemannian manifolds with symmetries \cite{BBN19}, sub-Riemannian structures on the sphere   $\mathbb{S}^7$ \cite{BL22}, and on quaternionic contact (qc) manifolds \cite{L21}. 
In the present paper we study the geometric interpretation of the second heat invariant for the intrinsic sub-Laplacian on $H$-type foliations. Introduced in \cite{BGRV21}, $H$-type foliations form good model spaces inside the family of sub-Riemannian manifolds with step two bracket generating distribution. An $H$-type foliation $\M$ carries an horizontal distribution $\Ho$ of constant rank with a totally geodesic and integrable complement $\Vr$. Moreover, $\M$ is the total space of a Clifford bundle adapted to the splitting $\Ho \oplus \Vr$ of the tangent space (see \cref{Section:H-type-foliations} for a precise definition). $H$-type foliations generalize previously studied examples such as Sasakian manifolds, Twistor spaces or 3-Sasakian manifolds and they carry a canonical connection (\emphs{Bott-connection}) which induces useful notions of horizontal and vertical curvature. We mention that the class of $H$-type foliations has a non-empty intersection with qc-manifolds, similar aspects of which have been recently studied by the   third author in \cite{L21}. 

The local models (tangent groups) of an H-type foliation $\M$ are H(eisenberg)-type Lie groups $G$ which were introduced by A. Kaplan in the 80th  in his study of fundamental solutions to hypoelliptic PDE, cf. \cite{K80}.  The corresponding $H$-type Lie algebras are in one-to-one correspondence with representations of Clifford algebras \cite{CDKR91}.  Such algebras are nilpotent of step two and the complement of their center (first layer) is naturally identified with a left-invariant bracket generating distribution which we call \emphs{horizontal}. Equipped with an inner product, which is induced by the metric on $\M$, the (tangent) $H$-type Lie algebras define the first algebraic and metric invariants of $\M$ as a sub-Riemannian manifold (see \cite{M08}).

$H$-type groups endowed with a left-invariant metric have the structure of a sub-Riemannian manifold themselves and subsequently have become an important class of examples in PDE \cite{BLU07,F73,MS15}, geometry and analysis \cite{BP08,KR95} or geometric measure theory, \cite{AKL09,FSS01}; see also the survey paper \cite{S03}. In analogy to known results in Riemannian geometry or the analysis in \cite{L21} we expect that for sufficiently well-behaved examples the second heat invariant of the sub-Laplacian can be expressed using a suitable concept of curvature.   We remark that the horizontal distribution on a step two nilpotent Lie group has vanishing curvature as well as vanishing second heat invariant. The class of $H$-type foliations, which we are considering in this paper, is more interesting when studying the effect of curvature terms on the heat kernel analysis of the sub-Laplacian. Roughly speaking, such manifolds may be thought of as ``curved versions'' of the $H$-type groups.

The task of expressing the second heat invariant of the sub-Laplacian on an $H$-type foliation in terms of curvature tensors requires the choice of a connection which is suitably adapted to the sub-Riemannian and Clifford bundle structure of $\M$. In this paper we choose the \emphs{Bott connection} in \cite[Chapter 5]{T88}, see also \cite{H12} which - different from the Levi-Civita connection in Riemannian geometry - has non-vanishing torsion. In case of a contact manifold or for pseudo-Hermitian or strictly pseudo-convex CR manifolds a natural connection is given by the \emphs{Tanno connection} \cite{T89} and the \emphs{Tanaka-Webster connection} \cite{T76,W78}, respectively. On qc-manifolds we may as well use the \emphs{Biquard connection} \cite{B00,L21}. 

We mention that all these examples have non-empty intersections with the class of $H$-type foliations. In these intersections the Bott connection coincides with the Tanno and the Tanaka-Webster connection, respectively,  and it differs from the Biquard connection by torsion properties. Different choices of a connection may lead to different representations of the second heat invariant \cite{B00}. 

Our heat kernel analysis is based on an appropriate choice of privileged coordinates on $\M$ constructed by Kunkel in \cite{FSS01} (see also \cite{JL89}) combined with recent methods from \cite{VHT21}. In fact, an inspection of  the proofs in \cite{VHT21} shows that the second heat invariant can be obtained as a convolution integral. The integrand is built from the heat kernel of the sub-Laplacian of the local nilpotent approximation (see \cite{BGG96} for an explicit formula) and an operator which, in particular, depends on the curvature tensor, see \cref{second_coef}. 

Our main result concerning the geometric interpretation of the second heat invariant for an $H$-type foliation is stated below as Theorem A   and will be explained more in detail next. 
  
Let $q \in \M$ and $t>0$. As is well-known the heat kernel $K_{\textup{sub}}$ of the intrinsic sub-Laplace operator $\Delta_{\textup{sub}}$ on $\M$ has a short time asymptotic 
expansion of the form
\begin{equation}
K_{\textup{sub}}(t; q,q)= 
t^{-\frac{Q}{2}}\Big{[} c_0(q) + c_1(q)t+  \ldots + c_N(q)t^N + O(t^{N+1})\Big{]}, \hspace{4ex} \mbox{\it as } \: t \downarrow 0. 
\end{equation}
Here $Q$ denotes the Hausdorff dimension of $\M$ as a metric space equipped with the {Carnot-Carath\'eodory metric}. 
The function $c_j(q)$ with $j=0,1,2 \ldots$ is referred to as the {\it $(j+1)$-th heat invariant} of the sub-Riemannian structure.

Let $\nabla$ denote the Bott connection on $\M$ with torsion $T$. Moreover, $\kh$ and $\tau$ are the \emphs{horizontal scalar curvature} and an invariant defined through the vertical bundle, respectively.  We note that an $H$-type foliation   $\M$ is said to have \emphs{horizontally parallel torsion} when $\nabla_\Ho T = 0$.

%%%%%%%%%%%%%%%%%%%%%%%%%%%%%%%%%%%%%%%%
\begin{inthm}\label{inthmA}
Let $(\M,\Ho,g_\Ho)$ be an H-type foliation with horizontally parallel torsion. Then the second heat invariant   $c_1(q)$ is a linear combination of the 
local geometric invariants $\kh$ and $\tau$:
\begin{equation}
  c_1=C_1\kh + C_2\tau,
\end{equation}
where $C_1$ and $C_2$ are universal constants depending only on $\rh$ and $\rv$, the ranks of the horizontal and vertical distribution, respectively. 
\end{inthm}

We address two further problems in connection with $H$-type foliations which so far seemed to not have been settled in the literature. The first question concerns the asymptotic behaviour of the   Popp volume $\textup{\it vol}$ of small balls   on $\M$. Let $r>0$ and consider the preimage $B(q,r) \subset \M$ around $q \in \M$ of a Korányi ball in $\mathbb{R}^{\rh+\rv}$ of radius $r$ under   (parabolic) privileged coordinates. In \cref{Section:Popp} it is shown: 

\begin{inthm}\label{inthmB}
Let $q \in \M$. As $r\to 0$ it holds:
$$r^{-Q} \cdot \mathrm{vol}(B(q,r))=a_{\rh,\rv}-b_{\rh,\rv}\kh(q)r^2+O(r^3),$$
where $a_{\rh,\rv}$ and $b_{\rh,\rv}$ are  suitable positive constants (which have explicit integral expressions). 
\end{inthm}

Finally, we give a sufficient and necessary condition for an $H$-type foliation with horizontally parallel torsion to be locally isometric as a sub-Riemannian manifold to its tangent group ($H$-type group). 

\begin{inthm}\label{inthmC}
Let $(\M,\Ho,g_\Ho)$ be an H-type foliation with horizontally parallel torsion. Then, $\M$ is locally isometric as a sub-Riemannian manifold to its tangent group if and only if the curvature tensor of the Bott connection vanishes, i.e. $R \equiv 0$.
\end{inthm}

The structure of the paper is as follows: In \cref{Section:H-type-foliations} we recall some basic notions in sub-Riemannian geometry and we introduce $H$-type foliations $\M$ which form a model class of sub-Riemannian manifolds with a Clifford bundle structure, \cite{BGRV21}. 

We recall the construction of privileged coordinates \cite{K08} in \cref{Section:Priviledged Coordinates}. Throughout the paper all local calculations will be performed within such coordinates which are well-adapted to the sub-Riemannian structure on $\M$.  In particular, the coordinate system obeys linear scaling in horizontal directions much like Riemannian normal coordinates, but obeys parabolic scaling in vertical directions reflecting the bracket structure. An essential ingredient to our analysis is the calculation of the homogeneous parts of the associated co-frame and connection one-forms in \cref{taylor_exp}.  Finally, \cref{inthmC} (\cref{thmC}) is achieved.

In \cref{Section:Popp} we calculate the Popp volume $\mathcal{P}$ on an $H$-type foliation which is used in the definition of the intrinsic sub-Laplacian $\Delta_{\mathrm{sub}}$.  A local expression of $\Delta_{\mathrm{sub}}$ is then given in \cref{sublap}.  \cref{Section:Popp} contains a proof of   \cref{inthmB} (\cref{thmB}) as its main result. 
             
Based on our calculations in local privileged coordinates and methods from \cite{VHT21} we introduce a new invariant $\tau$ in \cref{Section:Heat invariants} and derive a representation of the second heat invariant as a linear combination of $\tau$ and the horizontal scalar curvature $\kh$ with respect to the Bott connection. This result can be further refined when assuming horizontally parallel Clifford structure. The section concludes by proving \cref{thmA} (\cref{inthmA}).

The paper ends with an appendix which collects various technical calculations. 

%% file: sections/sec2_H-type-foliations.tex
\section{Sub-Riemannian geometry and H-type foliations} \label{Section:H-type-foliations}
In the setting of sub-Riemannian geometry, one works with a smooth manifold $\M$ of dimension $\rh + \rv$ equipped with a pair $(\Ho,g_\Ho)$ where  $\Ho$ is a rank $\rh$ subbundle of $T\M$ and $g_\Ho$ is a symmetric, positive-definite (2,0)-tensor on $\Ho$.  We insist that the \emphs{bracket-generating} condition holds, which we say is satisfied if at every point $p \in \M$ one can generate all of $T_p\M$ by taking sufficiently many Lie brackets of vector fields in $\s\Ho$ at $p$.  Notably, this is equivalent {\color{black} to} H\"ormander's condition in PDEs \cite{M02}. One can define the sub-Riemannian (Carnot-Carathéodory) distance $d_{cc}(p,q)$ for two points $p,q \in \M$ by the usual infimum formula taken over the space of smooth paths $C_{\Ho,p,q}$  connecting $p$ to $q$ such that $\dot\gamma(t) \in \Ho_{\gamma(t)}$ at almost every point along $\gamma$.  The now famous theorem of Chow \cite{C39} and Rashevsky \cite{R38} tells us that $(\M,d_{cc})$ is a metric space when $\Ho$ is bracket-generating.

Such a triple $(\M,\Ho,g_\Ho)$ is called a \emphs{sub-Riemannian manifold}.  These objects arise naturally in myriad settings and have been extensively studied in recent decades with increasing interest, see \cite{ABB20,M02,S86} for an overview.

An established approach is to study sub-Riemannian manifolds by equipping a Riemannian \emphs{penalty metric} 
$$g_\ve = g_\Ho \oplus \frac{1}{\ve}g_\Vr$$
parameterized by $\ve > 0$ on an appropriately chosen complementary distribution $\Vr$.  

\begin{notation}
In the sequel we will write $\Ho,\Vr$ as subscripts to denote appropriate projections of tensor fields. 
\end{notation}

In some key examples there is a natural choice of a complement, but generically this is not the case. However, when such a choice can be made one can then attempt to recover purely sub-Riemannian results by working with Riemannian tools on $(\M,g_\ve)$ and considering the limit $\ve \rightarrow 0^+$, however one should understand that Riemannian definitions of curvature do not make sense in this limit.

 %%%%%%%%%%%%%%%%%%%%%%%%%%%%%%%%%%%%%%%%%%%%%%%%%%%%%%%%%%%%%%%%%%%%%%%%%%%%%%%%%%%%%%%%
\subsection{H-type foliations}
We now review the notion of H-type foliations introduced in \cite{BGRV21} and motivated by earlier results of \cite{BG17}. In particular, there is a sub-class of H-type foliations associated in an essential way to qc-manifolds equipped with a totally-geodesic foliation, as appear in \cite{L21} which motivates the present paper.

Let $(\M,g)$ be a smooth, oriented, connected Riemannian manifold of dimension $\rh+\rv$.  To keep formulas shorter we will also use the notation $g(\cdot, \cdot)= \langle\cdot, \cdot \rangle$ for the Riemannian metric and occasionally we write $\|X\|^2=g(X,X)$. We assume that $\M$ is equipped with a Riemannian foliation with bundle-like complete metric and totally geodesic $\rv$-dimensional leaves. The subbundle $\Vr$ formed by vectors tangent to the leaves is referred to as the  \emphs{vertical distribution}. The subbundle $\Ho$ orthogonal to $\Vr$ will be called the \emphs{horizontal distribution}.

 We recall that the foliation is called \emphs{totally-geodesic} if the leaves of the foliation are totally-geodesic submanifolds, and said to have \emphs{bundle-like metric} if geodesics that are anywhere tangent to $\Ho$ remain always tangent to $\Ho$.  These conditions have the (respective) characterizations with respect to the Lie derivative \cite{T88} 
\begin{equation}
(\mathcal{L}_Z g)(X,X) = 0, \qquad (\mathcal{L}_X g)(Z,Z) = 0 
\end{equation}
for vector fields $X \in \Gamma(\Ho), Z \in \Gamma(\Vr)$. 

On such a manifold there is a canonical connection $\nabla$ that preserves the metric and the foliation structure, the \emphs{Bott connection}, see \cite{BGRV21,T88} and also a generalization to equiregular higher step sub-Riemannian manifolds, the \emphs{Hladky connection} \cite{H12}. 

\begin{defn}
The \defemphs{Bott connection} on a totally-geodesic foliation with bundle-like metric is uniquely characterized by the following properties:
\begin{enumerate}
\item (Metric) $\nabla g = 0$,
\item (Compatible) For $X \in \s{T\M}$, $\nabla_X \Ho \subseteq \Ho \text{ and } \nabla_X \Vr \subseteq \Vr.$
\item (Torsion) The torsion $T(X,Y):=\nabla_X Y-\nabla_Y X-[X,Y]$ satisfies 
	\begin{enumerate}
		\item $T(\Ho,\Ho) \subseteq \Vr,$
		\item $T(\Ho,\Vr)=T(\Vr,\Vr)=0.$
	\end{enumerate}
\end{enumerate}
\end{defn}

For every $Z\in\Gamma(\Vr)$ define a bundle endomorphism $J_Z:\Ho\rightarrow \Ho$ by
\begin{equation}
g(J_Z X,Y)=g(Z,T(X,Y)).
\end{equation}

\begin{lem}
Suppose the \defemphs{H-type condition}
\begin{equation}\label{H_type_cond}
J_Z^2=-\|Z\|^2 \Id_\Ho
\end{equation}
is satisfied.  
Then the tangent space $T_q\M$ at any point $q \in \M$ is generated by  $[X,\Ho]_q$ and $\mathcal{H}_q$ for every $X \in \s\Ho$ such that $X_q \ne 0$.
\end{lem}

Such a distribution is also called \emphs{strongly bracket generating} or \emphs{fat}, cf. \cite{M76,S86}.

\begin{proof}
Observe that the torsion of the Bott connection is given by
\begin{equation}
T(X,Y)=-[X_\Ho,Y_\Ho]_\Vr.
\end{equation}
Letting $X \in \s\Ho, Z \in \s\Vr$, we have the following relation for the bundle map $J_Z$
\begin{equation}
g\left(Z, T(X,J_ZX) \right) = g\left(J_ZX, J_ZX \right) = -g\left(X, J^2_ZX \right) = g\left( Z, \|X\|^2 Z \right).
\end{equation}

Combining the last 2 equalities we obtain
\begin{equation}
[X,J_ZX]_{\Vr}=-\|X\|^2 Z
\end{equation}
and the claim follows.
\end{proof}

\begin{defn}
If the H-type condition is satisfied then the sub-Riemannian manifold $(\M,\Ho,g_\Ho)$ is called an \defemphs{H-type foliation}.
\end{defn}

\begin{rem}
We observe that it can be at first counter-intuitive to call $(\M,\Ho,g_\Ho)$ a foliation as $\Ho$ is bracket generating, however we emphasize that the foliation tangent to $\Vr$ uniquely determines its $g$-complement $\Ho$; since we focus on these objects' structure as sub-Riemannian manifolds, we will continue to emphasize the sub-Riemannian distribution.  See also the related \cite[Remark 2.3]{BGRV21}.
\end{rem}

Note that due to polarization the H-type condition is equivalent to 
\begin{equation}\label{H-type-condition-polarization}
J_Z J_W + J_W J_Z= -2 g(Z,W) \Id_\Ho, \qquad  Z,W \in \s\Vr. 
\end{equation}

\subsection{Curvature and frames}
Let us denote by $R$ the \emphs{curvature tensor} of the Bott connection defined by
\begin{align}
%&T(X,Y):=\nabla_X Y-\nabla_Y X-[X,Y], \hspace{4ex} X,Y,Z \in \Gamma(T\M)\\
&R(X,Y)Z:=\nabla_X\nabla_Y Z-\nabla_Y \nabla_X Z-\nabla_{[X,Y]}Z. 
\end{align}

\begin{notation}
In the following we denote by $\left\{\theta^1,\cdots, \theta^\rh\right\}$ (resp. $\left\{\eta^1,\cdots, \eta^\rv\right\}$) the metric dual frame of a horizontal frame $\left\{X_1,\cdots,X_\rh\right\}$ of $\Ho$ (resp. vertical frame $\left\{Z_1,\cdots,Z_\rv\right\}$ of $\Vr$).\
 
It is occasionally convenient to have a notation for the entire frame and coframe. In these cases we may write $\{Y_1,\cdots,Y_{\rh+\rv}\}$ and $\{\nu^1,\cdots,\nu^{\rh+\rv}\}$ and one should interpret $Y_a = X_a$ for $a \in \{1,\cdots,\rh\}, Y_{\rh + a} = Z_a$ for $a \in \{1,\cdots,\rv\}$, and analogously for the $\nu^a$. Furthermore, in order to have a consistent index notation we will use different letters for different ranges of indices as follows:
\begin{equation}
a,b,c,d \in\{1,\cdots,\rh+\rv\},\hspace{2mm}\alpha,\beta,\gamma,\delta\in\{1,\cdots,\rh\},\quad i,j,k\in\{1,\cdots,\rv\}.
\end{equation}  
\end{notation}

With this convention we set
\begin{equation}
J_{\alpha\beta}^{i}:=g(J_{Z_i}X_\alpha,X_\beta),\quad T_{\alpha\beta}^{i}:=\eta^i\left(T(X_\alpha,X_\beta)\right)\text{ and }R_{abc}^d:=\nu^d \left(R(Y_a,Y_b)Y_c\right)
\end{equation}

and define the \emphs{horizontal scalar curvature} induced by the Bott connection 
\begin{equation}\label{Def: scalar curvature}
\kh = \sum_{\alpha, \beta=1}^\rh R_{\alpha \beta \beta}^{\alpha}.
\end{equation}

Throughout the text we frequently apply the \emphs{first Bianchi identity} for a connection with torsion, 
\begin{equation}
\circlearrowright R(X,Y)Z=\circlearrowright \nabla_X T(Y,Z) + \circlearrowright T(T(Y,Z),X),
\end{equation}
which in case of the Bott connection the torsion properties imply that this takes the simpler form
\begin{equation}\label{First_Bianchi_identity}
\circlearrowright R(X,Y)Z=\circlearrowright \nabla_XT(Y,Z) \in \Gamma(\Vr) 
\end{equation}
with $\circlearrowright$ denoting the cyclic sum. 

\begin{notation}\label{not:summations}
In the sequel we use a nonstandard summation convention: whenever any index is repeated twice it must be summed.  When unclear, we will make the summations explicit.  For example, \eqref{Def: scalar curvature} becomes
\begin{equation}
\kh = R_{\alpha \beta \beta}^{\alpha}.
\end{equation}
in this notation.

Take care to note that this is not the standard Einstein summation convention, wherein repeated indices are only summed when they appear as both a superscript and subscript.
\end{notation}

%%%%%%%%%%%%%%%%%%%%%%%%%%%%%%%%%%%%%%%%%%%%%%%%%%%%%%%%%%%%%%%%%%%%%%%%%%%%%%%%%%%%%%%%%

%% file: sections/sec3_privileged-coordinates.tex
\section{Construction of privileged coordinates}
\label{Section:Priviledged Coordinates}
%%%%%%%%%%%%%%%%%%%%%%%%%%%%%%%%%%%%%%%%%%%%%%%%%%%%%%%%%%%%%%%%%%%%%%%%%%%%%%%%%%%%%%%%%%%  

Consider $\R^{\rh+\rv}=\R^\rh \oplus \R^\rv$ with coordinates $(x,z)=(x^1,\ldots,x^\rh,z^1,\ldots,z^\rv)$ and the natural dilation $\delta_t \colon \R^{\rh+\rv} \rightarrow \R^{\rh+\rv}$ defined for $t > 0$ by 
\begin{equation} \label{def:dilation}
\delta_t(x,z) = (tx,t^2z).
\end{equation} 
We associate the following weights $w$ for the generators $1$, $x^{\alpha}$, $z^i$, $\frac{\partial}{\partial x^\alpha}$, $\frac{\partial}{\partial z^i}$ of polynomial vector fields:
\begin{equation}
w(1)=0,\quad w(x^{\alpha})=1,\quad w(z^i)=2,\quad w \left(\frac{\partial}{\partial x^\alpha} \right) =-1,\quad w \left(\frac{\partial}{\partial z^i}\right) =-2.
\end{equation}

A {\it formal monomial} is a product of the generators, the order of which is the sum of weights of each term. A polynomial vector field on $\R^{\rh+\rv}$ (considered as a differential operator of first order) is said to be \emphs{homogeneous} if it is a sum of monomials of the same weight. For any smooth vector field, applying the Taylor expansion for the coefficients, one can rearrange the terms into a sum of homogeneous terms
\begin{equation}
X \sim \sum_{l=-2}^{+\infty} X^{(l)},
\end{equation}
where $X^{(l)}$ is a homogeneous vector field of order $l$.  Written in this form, we will call this a \emphs{parabolic Taylor expansion}.

We say that a system of coordinates $\Psi \colon U \to \R^{\rh+\rv}$ centered at $q \in U \subseteq \M$ is \emphs{privileged} if for any horizontal vector field $X \in \s\Ho$ the terms of the homogeneous order $-2$ vanish, see more details in \cite[Chapter 10.4]{ABB20} or \cite{B96}.

Now we will construct a system of privileged coordinates that is convenient for the further calculations; it is reminiscent of normal Euclidean coordinates and normal coordinates for CR manifolds \cite{JL89,K08}. Let $q\in \M$ be fixed. Consider an arbitrary tangent vector $X+Z\in\Ho_q\oplus\Vr_q=T_q \M$ and define $\gamma_{X+Z}$ to be the curve starting at $q$ with 
  $$\gamma_{X+Z}(0)=q,\hspace{2mm} \dot{\gamma}_{X+Z}(0)=X,\hspace{2mm} \left. D_t \dot{\gamma}_{X+Z} \right\vert_{t=0} = Z \text{ and } D_t^2 \dot{\gamma}_{X+Z}=0,$$
  where $D_t$ denotes the covariant derivative along $\gamma_{X+Z}$. According to \cite{K08} (see also \cite{JL89}), there are neighborhoods $0\in O\subset T_q \M$ and $q\in U\subset \M$ so that the map $\varphi: O\longrightarrow U$ defined by
  $$\varphi(X+Z):=\gamma_{X+Z}(1)$$ 
  is a diffeomorphism, and satisfies the parabolic scaling $\varphi(tX+t^2Z)=\gamma_{X+Z}(t)$ whenever either side is defined.  Such a curve $\gamma_{X+Z}$ will be called a 
\emphs{parabolic geodesic}. 

Fix a horizontal (resp. vertical) orthonormal frame $\{X_1(q),\dots,X_\rh(q)\}$ (resp. $\{Z_1(q),\dots,Z_\rv(q)\}$) at $q$ and extend these vectors to be parallel  along parabolic geodesics 
 starting at $q$.  It is known (cf. \cite{K08}) that such extensions necessarily are smooth in a neighbourhood of $q \in \M$.  In this way, we obtain a smooth  local  frame $\{X_1,\dots,X_\rh, Z_1,\dots,Z_\rv\}$ for $T\M=\Ho\oplus\Vr$. Note that the corresponding dual frame  $\{\theta^1,\dots,\theta^\rh,\eta^1,\dots,\eta^\rv\}$ is also parallel along parabolic geodesics. In the following we call such a frame a \emphs{special frame}. 

We define a privileged coordinate system $(x,z)$ on $U$ by composing the inverse of $\varphi$ with the map $\lambda: T_q\M\longrightarrow \R^{\rh+\rv}$ defined by $\lambda(W):=(\theta^\alpha(W),\eta^i(W))$  (cf. \cite[Example 10.31 (ii)]{ABB20}).  Note that if we consider another horizontal (resp. vertical) frame at $q$, then these frames are related  via an orthogonal transformation. Hence, the resulting coordinates are related  via a block diagonal matrix with orthogonal entries, i.e. a matrix of the form
\begin{equation}
\begin{pmatrix}
A&0\\
0&B
\end{pmatrix} \in {\bf O}(\rh+\rv)
\end{equation}
with $A\in {\bf O}(\rh)$ and $B\in {\bf O}(\rv)$,  where ${\bf O}(k)$ with $k \in \mathbb{N}$ denotes the group of orthogonal $k \times k$ real matrices. 
  
From now on we will work with this special frame and co-frame and frequently write
\begin{equation}
J_i:=J_{Z_i}\text{ for }i=1,\cdots,\rv.
\end{equation}

The infinitesimal generator $\tilde{P}$ of the $\R^+$-action $\delta_t$ (\cref{def:dilation}) on $\R^{\rh+\rv}$ is given in these coordinates by
\begin{equation}
\tilde{P}_{(x,z)} := \sum_{\alpha=1}^{\rh}x^\alpha\frac{\partial}{\partial x^\alpha}+2\sum_{i=1}^{\rv}z^i\frac{\partial}{\partial z^i}.
\end{equation}
We adapt the notion of homogeneity to tensor fields on $\M$. We will denote by $P$ the pullback of $\tilde{P}$ to $\s{T\M}$ and say that a tensor field  $\Theta$ is \emphs{homogeneous of order} $l = \ord(\Theta) \in \mathbb{Z}$ with respect to the above dilations if
\begin{equation}\label{order}
\Lie_P(\Theta)=l\Theta,
\end{equation}

\begin{prop}\label{order_pro}
Whenever they make sense, the following properties hold.
\begin{enumerate}
\item Let $X$ and $Y$ be homogeneous vector fields. Then $[X,Y]$ is homogeneous of order $\ord(X)+\ord(Y)$.
\item Let $X$ (resp. $\nu$) be a homogeneous vector field (resp. $1$-form). Then $\nu(X)$ is homogeneous of order $\mathrm{ord}(\nu	) + \mathrm{ord}(X)$.
\item Let $f$ (resp. $X$) be a homogeneous function (resp. vector field). Then $fX$ is homogeneous of order $\ord(f)+\ord(X)$.
\end{enumerate}
\end{prop}

We now prove a lemma crucial for our asymptotic analysis on H-type foliations, (cf. \cite{JL89,K08}):
\begin{lem}\label{generator}
In the above privileged coordinates, it holds:
\begin{equation}
\theta^\alpha(P) = x^\alpha, \eta^i(P), =z^i \text{ and }\omega^a_b(P)=0.
\end{equation}
Here $\omega^a_b$ denotes the connection $1$-forms w.r.t. the constructed special frame, i.e. 
\begin{equation}
\nabla X_b = X_a \otimes \omega^a_b. 
\end{equation}
In particular, at $p \in \M$ with the coordinates $(x,z) = (x^1,\dots,x^\rh,z^1,\dots,z^\rv) \in \R^{\rh+\rv}$ we can write $$P_p = \sum_{\alpha=1}^{\rh} x^\alpha X_\alpha+\sum_{i=1}^{\rv} z^i Z_i.$$
\end{lem}

\begin{proof} It will suffice to prove the statement along parabolic geodesics. Construct a privileged coordinate system on a neighborhood $U$ around $q \in \M$ as explained above, and denote the coordinate map by $\Psi := \lambda \circ \varphi^{-1} \colon U \rightarrow \R^{\rh+\rv}$. 

Let $$W := \sum x^\alpha X_\alpha(q) + \sum z^iZ_i(q) \in T_q\M,$$ be the vector written in the special frame $\{X_1,\cdots,X_\rh, Z_1,\cdots,Z_\rv\}$, where $(x,z)=\Psi(p)\in \R^{\rh+\rv}$.
The parabolic geodesic generated by $W$ can be written in local coordinates as
\begin{equation}
\Psi(\gamma_W(t)) = \lambda(tX(q) + t^2Z(q)) = (tx,t^2z) = \delta_t (x,z).
\end{equation}

By definition
\begin{equation}
\tilde{P}_{\Psi(\gamma_W(t))} = \sum t x^\alpha \frac{\partial}{\partial x^\alpha} + \sum 2t^2 z^i \frac{\partial}{\partial z^i}
\end{equation}
and so
\begin{equation}
\begin{split}\label{eq: P}
\dot\gamma_W(t) &= \frac{d}{dt} (\Psi^{-1} \circ \delta_t)(x,z) \\
&= \Psi^* \left( \sum x^\alpha \frac{\partial}{\partial x^\alpha} + \sum 2t z^i \frac{\partial}{\partial z^i} \right)_{\Psi(\gamma_W(t))} \\
&= t^{-1} (\Psi^*\tilde{P})_{\Psi(\gamma_W(t))} \\
&= t^{-1} P_{\gamma_W(t)}.
\end{split}
\end{equation}

We now observe that for any $Y$ along the geodesic $\gamma_W(t)$ in the special frame one has
\begin{equation}
\frac{d}{dt} \langle D_t\dot\gamma_W, Y \rangle = (\nabla_{\dot\gamma_W} g) (D_t\dot\gamma_W, Y) + \langle D^2_t\dot\gamma_W, Y \rangle + \langle D_t\dot\gamma_W, D_tY \rangle = 0
\end{equation}
and so it follows from the initial condition that $D_t\dot\gamma_W = Z$ is parallel, or equivalently the vector field $\Psi_*Z$ has constant coefficients in $\R^{\rh+\rv}$.  For any 1-form $\nu$ parallel along $\gamma_W$ we have that
\begin{equation}
\frac{d}{dt}(\nu(\dot\gamma_W(t))) = \nu(D_t\dot\gamma_W(t)) = \nu(Z).
\end{equation}
We recover a system of equations for the special co-frame $\{\theta^\alpha,\eta^i\}$
\begin{equation}
\begin{cases} 
\frac{d}{dt}(\theta^\alpha(\dot\gamma_W(t))) &= 0 \\
\theta^\alpha(\dot{\gamma}_W(0)) &= x^\alpha
\end{cases}
\qquad
\begin{cases} 
\frac{d}{dt}(\eta^i(\dot\gamma_W(t))) &= z^i \\
\eta^i(\dot{\gamma}_W(0)) &= 0
\end{cases}
\end{equation}
from which we obtain the expressions
\begin{equation}
\theta^\alpha(\dot\gamma_W(t)) = x^\alpha, \qquad
\eta^i(\dot\gamma_W(t)) = tz^i.
\end{equation}

Applying this with \eqref{eq: P}, 
\begin{equation}
\theta^\alpha(P_{\gamma_W(t)}) = tx^\alpha, \qquad
\eta^i(P_{\gamma_W(t)}) = t^2z^i,
\end{equation}
and so we can conclude
\begin{equation}
P_{\gamma_W(t)} = \sum_{\alpha=1}^{\rh} t x^\alpha X_\alpha+\sum_{i=1}^{\rv} t^2 z^i Z_i.
\end{equation}
In particular, for $p \in U$ with parabolic coordinates $(x,z)$, we have
\begin{equation}
P_p = P_{\gamma_W(1)} = \sum_{\alpha=1}^{\rh} x^\alpha X_\alpha+\sum_{i=1}^{\rv} z^i Z_i
\end{equation}
as desired.

For the connection forms, observe that
$$
\omega^a_b(P)=g(X_{b},\nabla_{P}X_a)=tg(X_{b},D_{t}X_a)
$$
follows from the fact $P=t\dot\gamma_W$ for some $t>0$ with initial vector $W$. Since the frame is parallel along any parabolic geodesic $\gamma_W$, we obtain that $D_{t}X_a=0$ which concludes the proof.
\end{proof}

As mentioned in~\cite{JL89}, if $\nu$ is a differential form, then the homogeneous part $\nu^{(l)}$ of degree $l>0$ of its parabolic Taylor expansion can be computed by the formula
\begin{equation}\label{order_comp}
\nu^{(l)}=\frac{1}{l}\left(\Lie_P(\nu)\right)^{(l)}=\frac{1}{l}\left(P\lrcorner \hspace{0,5mm}d\nu+d(P\lrcorner\hspace{0,5mm} \nu)\right)^{(l)}\text{ for }l\geq 1.
\end{equation}
 Furthermore, using the structure equations (see \cref{Prop_1_Appendix} in the 
Appendix) of the Riemannian foliation $\M$ (w.r.t. the special frame and co-frame):
\begin{equation}\label{GL_expansion_differentials}
\begin{cases}
d\theta^\alpha &= \theta^\beta \wedge \omega_\beta^\alpha, \\
d\eta^i &= \frac{1}{2} J^i_{\alpha\beta} \theta^\alpha \wedge \theta^\beta + \eta^j \wedge \omega_j^i, \\
d\omega_a^b &= \frac{1}{2} R_{\alpha\beta a}^b \theta^\alpha \wedge \theta^\beta + R_{\alpha ia}^b \theta^\alpha \wedge \eta^i + \frac{1}{2} R_{jka}^b \eta^j \wedge \eta^k + \omega_a^c \wedge \omega_c^b. \\
\end{cases}
\end{equation}

We then obtain exact expressions for the homogeneous parts of the special co-frame and connection $1$-forms.
\begin{prop} \label{taylor_exp}
In privileged coordinates $(x^\alpha,z^i)$, we have $\theta^{\alpha(0)} = \eta^{i(0)} = 0$ and the higher order components and connection $1$-forms are given for $l \geq 1$ by
\begin{equation}
\begin{cases} \theta^{\alpha(l)} &= \frac{1}{l}\left( dx^\alpha + x^\beta \omega_{\beta}^{\alpha} \right)^{(l)}, \\
\eta^{i(l)} &= \frac{1}{l} \left( dz^i + z^j \omega_{j}^{i} + J_{\alpha\beta}^{i} x^\alpha \theta^\beta \right)^{(l)}, \\
\omega_{a}^{b(l)} &= \frac{1}{l} \left( R_{\alpha\beta a}^{b} x^\alpha \theta^\beta + R_{\alpha ia}^{b} x^\alpha \eta^i + R_{i\alpha a}^{b} z^i \theta^\alpha + R_{jka}^{b} z^j \eta^k \right)^{(l)}. \\
\end{cases}
\end{equation}
\end{prop}

\begin{proof} 
 The identity $\theta^{\alpha(0)} = \eta^{i(0)} = 0$ follows from the observation that  the one forms $\theta^\alpha$ and $\eta^i$ are linear combinations of $dx^\alpha$ and $dz^i$ 
which are of homogeneous order one and two, respectively. The remaining identities follow from \cref{GL_expansion_differentials} and \cref{order_comp}. 
\end{proof}

Applying the proposition, we recover the particular expressions.
\begin{table}[h]
\centering
\begin{tabular}{|c||c|c|c|}
    \hline
    $l$ & $\theta^{\alpha(l)}$ & $\eta^{i(l)}$ & $\omega^{b(l)}_a$\\
    \hline \hline
    1 & $dx^{\alpha}$ & 0 & 0 \\
    \hline
    2 & $0$ & $\frac{1}{2}(dz^i + J_{\alpha \beta}^i(q) x^{\alpha} dx^{\beta})$ & $\frac{1}{2} R_{\alpha \beta a}^b(q) x^{\alpha} dx^{\beta}$ \\
    \hline
    3 & $\frac{1}{6} R_{\gamma \delta \beta}^{\alpha}(q) x^{\beta} x^{\gamma} dx^{\delta}$ & $\frac{1}{3} J_{\alpha \beta}^{i(1)}x^{\alpha} dx^{\beta}$ & \\
    \hline
    4 & $\frac{1}{4} (z^j \omega_{j}^{i(2)} + J_{\alpha \beta}^i(q) x^{\alpha} \theta^{\beta(3)} + J_{\alpha \beta}^{i(2)} x^{\alpha} dx^{\beta})$ & & \\
    \hline
\end{tabular}
\end{table}

The following properties will be useful in the sequel.
 \begin{lem}\label{properties} In privileged coordinates $(x^\alpha,z^i)$, it holds:
 \begin{enumerate}
 \item The  homogeneous terms of $J_{\alpha\beta}^{i}$ of order one and two are given by
  \begin{align}
  J_{\alpha \beta}^{i(1)}&=x^\gamma X_\gamma(J_{\alpha\beta}^{i})(q), \\
  J_{\alpha \beta}^{i(2)}&=\frac{1}{2}\left(x^\gamma x^\delta X_\gamma X_\delta(J_{\alpha\beta}^{i})(q)+z^j Z_j(J_{\alpha\beta}^{i})(q)\right).
   \end{align}
  \item The connection $1$-forms $\omega_{a}^{b}$ vanish at the point $q$.
 \item The horizontal derivatives of $J_{\alpha\beta}^{i}$ at $q$ are related to the torsion $T$ by
 $$X_\gamma(J_{\alpha\beta}^{i})(q)=g_q\big{(}Z_i,(\nabla_{X_\gamma}T)(X_\alpha,X_\beta)\big{)}.$$
 \end{enumerate}
 \end{lem}
 \begin{proof}
(1):  Let $f$ be a smooth function near $q\in \M$ and consider its  parabolic Taylor expansion around $q$ in privileged coordinates $(x^\alpha,z^i)$, namely: 
 $$f=f(q)+\sum_{l>0}f^{(l)},$$
 where $f^{(l)}$ denotes the $l$-homogeneous part of $f$. Applying the Lie derivative $\Lie_P$ to both sides of this expansion we find
  $$f^{(l)}=\frac{1}{l}\big{(}P(f)\big{)}^{(l)},\text{ for } l\geq 1.$$
 According to \cref{generator} and choosing $f=J_{\alpha\beta}^{i}$ we obtain
 $$J_{\alpha\beta}^{i(l)}=\frac{1}{l}\left(P(J_{\alpha\beta}^{i})\right)^{(l)}.$$
 Hence, it follows that
 \begin{align}
 J_{\alpha\beta}^{i(1)}=\left(P(J_{\alpha\beta}^{i})\right)^{(1)}
 =\big{(}x^\gamma X_\gamma(J_{\alpha\beta}^{i})+z^j Z_j(J_{\alpha\beta}^{i})\big{)}^{(1)}
 =x^\gamma X_\gamma(J_{\alpha\beta}^{i})(q) 
 \end{align}
 and similarly 
 \begin{align}
 J_{\alpha\beta}^{i(2)}&= \frac{1}{2}\big{(}x^\gamma X_\gamma(J_{\alpha\beta}^{i})+z^j Z_j(J_{\alpha\beta}^{i})\big{)}^{(2)} \\
 &=\frac{1}{2}\big{(}x^\gamma \big(X_\gamma(J_{\alpha\beta}^{i})\big)^{(1)}+z^j Z_j(J_{\alpha\beta}^{i})(q)\big{)}\\
 &=\frac{1}{2}\big{(}x^\gamma x^\delta X_\gamma X_\delta(J_{\alpha\beta}^{i})(q)+z^j Z_j(J_{\alpha\beta}^{i})(q)\big{)}.
 \end{align}

 (2): In privileged coordinates $(x^\alpha,z^i)$ we write $\omega_{a}^{b}=C_{a\alpha}^{b}dx^\alpha+D_{a i}^{b}dz^i$ 
 with some smooth functions $C_{a \alpha}^{b}$ and $D_{a i}^{b}$. Then, we obtain
 $$\omega_{a}^{b(1)}=C_{a \alpha}^{b}(q)dx^\alpha \hspace{2ex} \text{ and }\hspace{2ex} \omega_{a}^{b(2)}=C_{a \alpha}^{b(1)}dx^\alpha+D_{a i}^{b}(q)dz^i.$$
 However, according to \cref{taylor_exp} we have $\omega_{a}^{b(1)}=0$ and $\omega_{a}^{b(2)}=\frac{1}{2}R_{\alpha\beta a}^{b}(q)x^\alpha dx^\beta.$
 In particular, this shows that
 $C_{a \alpha}^{b}(q)=D_{a i}^{b}(q)=0,$
 i.e. $\omega_{a}^{b}(q)=0$.

 (3):  Since $\nabla$ is a metric connection, we have: 
\begin{align}
X_{\gamma}(J_{\alpha \beta}^i)
&= X_{\gamma} g\big{(} J_{Z_i}X_{\alpha}, X_{\beta} \big{)} = X_{\gamma} g\big{(} Z_i, T(X_{\alpha}, X_{\beta}) \big{)}\\
&= g \big{(} \nabla_{X_{\gamma}}Z_i,  T(X_{\alpha}, X_{\beta}) \big{)}+ g\big{(} Z_i, \nabla_{X_{\gamma}}(T(X_{\alpha}, X_{\beta})) \big{)}. 
\end{align}
At the point $q \in \M$ we have $\omega_{a}^{b}(q)=0$ according to (2) and therefore, 
\begin{equation}
(\nabla_{X_{\gamma}}X_\alpha)_q=(\nabla_{X_{\gamma}}X_\beta)_q =(\nabla_{X_{\gamma}}Z_i)_q=0
\end{equation}
showing the assertion. 
 \end{proof}

Now we compute the low order homogeneous parts in the parabolic Taylor series of the special frame $\{X_1,\ldots,X_\rh,Z_1,\ldots,Z_\rv\}$. For this, we first define the following linearly independent vector fields on $\R^{\rh+\rv}$:
\begin{equation}
\hat{X}_\alpha := \frac{\partial}{\partial x^\alpha} + J_{\alpha\beta}^i(q) x^\beta \frac{\partial}{\partial z^i} \text{ and } \hat{Z}_i := 2\frac{\partial}{\partial z^i}.
\end{equation}

Note that $\hat{X}_\alpha$ (resp. $\hat{Z}_i$) is homogeneous of degree $-1$ (resp. $-2$).  

\begin{rem}\label{Remark_3_5}
The advantage of the chosen privileged coordinates is that the nilpotent Lie algebra (that depends on the choice of privileged coordinates) will coincide with the $H$-type Lie algebra, which is the first algebraic and metric invariant at each  point $q \in \M$ as mentioned in the introduction. Moreover, the vector fields $\hat{X}_\alpha$ and $\hat{Z}_i$ are left invariant vector fields on the corresponding ($H$-type) Lie group $\mathbb{G}(q)$ with group law ``$\star$'' as in e.g. \cite{B96}.
\end{rem}

\begin{lem}\label{expansion}
In privileged coordinates $(x^\alpha,z^i)$, for the horizontal frame we can express the low order homogeneous parts as
\begin{equation}
X_\alpha^{(-1)}=\hat{X}_\alpha,\hspace{2mm}X_\alpha^{(0)}=\frac{1}{3}J_{\alpha \beta}^{i(1)}x^\beta \hat{Z}_i\; \text{ and }\; X_\alpha^{(1)}=f^\beta_\alpha \hat{X}_\beta+h^i_\alpha \hat{Z}_i,
\end{equation}
where 
\begin{align}
f^\beta_\alpha &= \frac{1}{6} x^\gamma x^\delta R_{\alpha\gamma\delta}^{\beta}(q), \\
h^i_\alpha &= \frac{1}{8} R_{\alpha\beta j}^{i}(q) z^j x^\beta + \frac{1}{24} J_{\beta\gamma}^i(q) R_{\alpha \beta^\prime \gamma^\prime}^\gamma (q) x^\beta x^{\beta^\prime} x^{\gamma^\prime} + \frac{1}{4} x^\beta J_{\alpha\beta}^{i(2)}. \\
\end{align}

Similarly, for the vertical frame it holds that 
\begin{equation} 
Z_i^{(-2)} = \hat{Z}_i \text{ and } Z_i^{(-1)} = Z_i^{(0)} = 0.
\end{equation} 
\end{lem}

\begin{proof} 
Locally near  $0\in \R^{\rh+\rv}$ we consider the expansion:
\begin{equation}\label{frame}
X_\alpha = s_\alpha^\beta \hat{X}_\beta + r_\alpha^j \hat{Z}_j, 
\end{equation}
where $s_\alpha^\beta$ and $r_\alpha^{j}$ are smooth functions. The parabolic Taylor expansion of $s_\alpha^\beta$ and $r_\alpha^j$ around $0$ have the form: 
\begin{align}
s_\alpha^\beta &= s^{\beta(0)}_\alpha + s^{\beta(1)}_\alpha + s^{\beta(2)}_\alpha + \cdots \\
r_ \alpha^j &= r^{j(0)}_\alpha + r^{j(1)}_\alpha + r^{j(2)}_\alpha + \cdots.
\end{align}

Here $s^{\beta(l)}_\alpha$ (resp. $r^{j(l)}_\alpha$) (for $l \geq 0$) denotes the homogeneous part of order $l$ in the expansion. Applying $\theta^\gamma$ (resp. $\eta^i$) to both sides of \cref{frame} and taking the homogeneous parts of order $l$, we obtain recursive formulas for $s^{\gamma(l)}_{\alpha}$ and $r^{i(l)}_{\alpha}$ (for $l\geq 1$): 
\begin{align}
s^{\gamma(l)}_\alpha &= -\sum_{m=0}^{l-1} s^{\beta(m)}_\alpha \theta^{\gamma(l-m+1)}(\hat{X}_\beta) - \sum_{m=0}^{l} r^{j(m)}_\alpha \theta^{\gamma(l-m+2)}(\hat{Z}_j), \\
r^{i(l)}_\alpha &= -\sum_{m=0}^{l-1} s^{\beta(m)}_\alpha \eta^{i(l-m+1)} (\hat{X}_\beta) - \sum_{m=0}^{l-1} r^{j(m)}_\alpha \eta^{i(l-m+2)} (\hat{Z}_j)
\end{align}
with initial conditions $s^{\gamma(0)}_\alpha = \delta_{\alpha\gamma}$ and $r^{i(0)}_\alpha = 0$.  Combining these identities with the expressions for the homogeneous components of $\theta^\gamma$ and $\eta^i$ in \cref{taylor_exp} we find: 
\begin{equation}
s^{\beta(1)}_\alpha = 0, \quad r^{j(0)}_\alpha = r^{j(1)}_\alpha=0,\text{ and } r^{j(2)}_\alpha = \frac{1}{3} J_{\alpha\beta}^{j(1)} x^\beta.
\end{equation}
Hence, it follows that
\begin{equation}
s^{\beta(2)}_{\alpha}=-\theta^{\beta(3)}(\hat{X}_\alpha)=\frac{1}{6}R_{\alpha\gamma\delta}^{\beta}(q)x^\gamma x^\delta,
\end{equation}
and
\begin{align}
r^{j(3)}_\alpha &= -\eta^{j(4)}(\hat{X}_\alpha) \\
&= \frac{1}{8}R_{\alpha\beta i}^{j}(q)z^i x^\beta + \frac{1}{24}J_{\beta\gamma}^j(q) R_{\alpha\beta^\prime\gamma^\prime}^\gamma(q) x^\beta x^{\beta^\prime} x^{\gamma^\prime} + \frac{1}{4}x^\beta J_{\alpha\beta}^{j(2)}.
\end{align}

Now the assumption follows from the formula ($l\geq 1$)
\begin{equation}
X_\alpha^{(l)} = s_\alpha^{\beta(l+1)} \hat{X}_\beta + r_\alpha^{j(l+2)} \hat{Z}_j.
\end{equation}

The identities $Z_i^{(-2)}=\hat{Z}_i$ and $Z_i^{(-1)}=Z_i^{(0)}=0$ follow by a similar calculation.
\end{proof}

We constructed a graded step two nilpotent Lie algebra generated by the vector fields $X_1^{(-1)},\cdots,X_\rh^{(-1)}$ on $\R^{\rh+\rv}$; as was mentioned, this Lie algebra coincides with an $H$-type Lie algebra. The corresponding connected, simply connected Lie group is the metric tangent group (or cone), see \cite{B96} or \cite[Proposition 10.77]{ABB20}.\

The next theorem shows that the curvature tensor $R$ of the Bott connection is exactly the obstruction of the $H$-type foliation (with horizontal parallel torsion,  that is $\nabla_X T=0$ for all $X \in \s\Ho$) to be locally isometric to an $H$-type group in the sense of sub-Riemannian manifolds. We recall the definition (see \cite{BR13}): 
\vspace{1mm}\par 
Let $(\M_j, \mathcal{H}_j, \langle \cdot ,\cdot \rangle_j)$ for $j=1,2$ be sub-Riemannian manifolds. We call a smooth map $$\varphi: \M_1\rightarrow\M_2$$ {\it horizontal} if $\varphi_*(\mathcal{H}_1) \subset \mathcal{H}_2$. 

Moreover, $\varphi$ is called a (local) sub-Riemannian isometry, if it is a horizontal (local) diffeomorphism such that 
$\phi_*: (\mathcal{H}_1, \langle \cdot, \cdot \rangle_1) 
\longrightarrow (\mathcal{H}_2, \langle \cdot, \cdot \rangle_2)$ is an isometry.

\begin{thm} \label{thmC}
Let $\M$ be an H-type foliation with horizontally parallel torsion. Then $\M$ is locally isometric (as a sub-Riemannian manifold) to its tangent group if and only if the curvature tensor of the Bott connection vanishes, i.e. $R=0$.
\end{thm}
\begin{proof}
It suffices to construct a local isometry between $\M$ and its tangent group under the condition $R\equiv 0$.
Assume that the curvature tensor $R$ is identically zero. Then, due to the properties of the Bott connection we can construct near every fixed point $q\in \M$, a parallel orthonormal frame $\{X_1,\cdots,X_n,Z_1,\cdots,Z_m\}$, where $X_1,\cdots,X_n$ (resp. $Z_1,\cdots,Z_m$) are horizontal (resp. vertical) vector fields. Next, we look at the Lie bracket relations of this frame. Then it holds that
\begin{align} \textstyle
	[X_\alpha,Z_i] &= -T(X_\alpha,Z_i) = 0, \\
	[Z_i,Z_j] & = -T(Z_i,Z_j) = 0,\\
	[X_\alpha,X_\beta] &= -T(X_\alpha,X_\beta) = -\textstyle \sum_{k=1}^\rv J_{\alpha\beta}^k Z_k.
\end{align}

Since the torsion is horizontally parallel, it follows that the functions $J_{\alpha\beta}^k$ are in fact constant. Now we consider the structure equations (\ref{GL_expansion_differentials}) with respect to this frame and the associated co-frame $\theta^1,\cdots,\theta^\rh,\eta^1,\cdots,\eta^\rv$:
\begin{equation}\label{struc_equ}
d\theta^\alpha=0 \text{ and } d\eta^i=\frac{1}{2}J^{i}_{\alpha\beta}\theta^\alpha \wedge \theta^\beta.
\end{equation}
Then, for every $\alpha$ we can find a locally defined smooth function $x^\alpha$ such that $\theta^\alpha=dx^\alpha$. Inserting this in \cref{struc_equ} we can find a smooth function $z_i$ such that
\begin{equation}
\eta^i = dz^i + \frac{1}{2}J^i_{\alpha\beta} x^\alpha dx^\beta.
\end{equation}

Finally, the map $p \mapsto (x^1(p),\cdots,x^n(p),z^1(p),\cdots,z^m(p))$ defines a local sub-Riemannian isometry between $\M$ and its tangent $H$-type group.
\end{proof}

%% file: sections/sec4_popp.tex
\section{Popp measure and volume of small parabolic geodesic balls}
\label{Section:Popp}
%%%%%%%%%%%%%%%%%%%%%%%%%%%%%%%%%%%%%%%%%%%%%%%%%%%%%%%%%%%%%%%%%%%%%%%%%%%%%%%%%%%%%%%%%%%%%%%%%
Let $\{X_1,\cdots,X_{\rh},Z_1,\cdots,Z_\rv\}$ be a special frame near $q\in \M$ in the sense of  \cref{Section:Priviledged Coordinates}. We observe that this frame is also an \emphs{adapted frame} for the sub-Riemannian structure of the H-type Riemannian foliation $\M$ in the sense of \cite{BR13}; that is, $\{X_1, \cdots, X_\rh\}$ spans the horizontal distribution while $\{Z_1, \cdots, Z_\rv\}$ spans the vertical distribution. It follows from the results therein that the Popp measure $\mathcal{P}$ associated to the corresponding equiregular sub-Riemannian structure can be expressed locally in the form 
\begin{equation}
\mathcal{P}=\frac{1}{\sqrt{\det{B}}}\theta^1\wedge\cdots\wedge\theta^{\rh}\wedge\eta^1\wedge\cdots\wedge\eta^\rv.
\end{equation}
Here $B=(B_{ij})_{ij}$ is the $(\rv\times \rv)$-matrix function locally defined near $q$ with coefficients given by
\begin{equation}
B_{ij}:=\sum_{\alpha,\beta=1}^{\rh}b_{\alpha\beta}^{i}b_{\alpha\beta}^{j},
\end{equation}
where $b_{\alpha\beta}^{i}$ are defined for $\alpha,\beta=1,\cdots,\rh$ and $i=1,\cdots,\rv$ by
\begin{equation}
b_{\alpha\beta}^{i}:=g\big{(}[X_\alpha,X_\beta],Z_i\big{)}= -g(J_iX_\alpha,X_\beta).
\end{equation}

 %Now by definition of the almost complex structures $(J_i)_i$ we can write:
 %\begin{align}
 %b_{\alpha\beta}^{i}&=g([X_\alpha,X_\beta],Z_i)\\
% &=g(J_iX_\alpha,X_\beta).
% \end{align}
Hence it follows that
\begin{equation}
B_{ij} = \sum_{\alpha,\beta=1}^{\rh} g(J_iX_\alpha, X_\beta) g(J_jX_\alpha, X_\beta) = \sum_{\alpha=1}^\rh g(J_iX_\alpha, J_j X_\alpha) = \rh\delta_{ij}.
\end{equation}
In the last equality we used the skew-symmetry of the almost complex structures $J_i$ together with the $H$-type condition. This shows that $B$ is a diagonal matrix:
\begin{equation}
B= \rh \cdot \Id \in \R^{\rv \times \rv}
\end{equation}
and hence we obtain a formula for the Popp measure:
\begin{lem}\label{measure} 
The Popp measure $\mathcal{P}$ for the H-type Riemannian foliation $(\M,\Ho,g_\Ho)$ has the form 
\begin{equation}
\mathcal{P}=\rh^{-\rv/2}\omega,
\end{equation}
where $\omega := \theta^1 \wedge \cdots \wedge \theta^\rh \wedge \eta^1 \wedge \cdots \wedge \eta^\rv$.
\end{lem}

Let us denote by $\Psi \colon q \in U \subseteq \M \rightarrow \R^{\rh+\rv}$ the (privileged) coordinate map and define the parabolic ball with center $q$ and (small) radius $r>0$ by 
\begin{equation}
B(q,r) := \Psi^{-1} \left(\hat{B}(0,r)\right),
\end{equation}
where 
\begin{equation}
\hat{B}(0,r) := \left\{ (x,z) \in \R^{\rh+\rv} \colon \|(x,z)\|_h := (|x|^4+|z|^2)^{1/4} \leq r \right\}
\end{equation}
is the {\it Kor\'anyi ball} of radius $r>0$ centered at $0 \in \mathbb{R}^{n+m}$. Here $| \cdot |$ denotes the Euclidean inner product on $\R^\rh$ and $\R^\rv$. 

\begin{rem}
Note that a parabolic ball is independent of the choice of an orthonormal frame at $q$ and that according to the ball-box theorem (see \cite[Theorem 2.10]{M02}) the sub-Riemannian distance $d_{cc}$ in the privileged coordinates $(x,z)$ is equivalent to the homogeneous distance $d_h$ defined by 
\begin{equation}
d_h \left( (x,z), (x^\prime, z^\prime) \right) := \|(x,z)^{-1 } \star (x^\prime,z^\prime)\|_h
\end{equation}
for $(x,z),(x^\prime,z^\prime) \in \R^{\rh+\rv} \simeq \mathbb{G}(q)$. Here $\star$ denotes the group law on $\mathbb{G}(q)$, as in \cref{Remark_3_5}.
\end{rem}

The next lemma provides a geometric interpretation of the scalar curvature induced by the Bott connection:
\begin{thm}\label{thmB}
Let $q \in \M$. As $r\to 0$ it holds:
\begin{equation}
r^{-Q} \cdot \mathrm{vol}(B(q,r)) = a_{\rh,\rv} - b_{\rh,\rv}\kh(q) r^2 + O(r^3),
\end{equation}
where $a_{\rh,\rv}$ and $b_{\rh,\rv}$ are suitable positive constants (which will be given in the proof below) 
and $Q=n+2m$ is the Hausdorff dimension of the metric space $(\M, d_{cc})$.
\end{thm}

\begin{proof}
We write 
\begin{equation}
\mathrm{vol}(B(q,r)) = \frac{1}{\rh^{\rv/2}} \int_{B(q,r)} \omega
= \frac{1}{\rh^{\rv/2}} \int_{\hat{B}(0,1)} \hor,
\end{equation}
where we denote the pushforward to $\R^{\rh + \rv}$ of $\omega$ as $\ho = \Psi_*\omega $ and its pullback by $\delta_r$ as $\hor = \delta_r^* \ho$.

By \cref{taylor_exp}, the low order homogeneous parts in the expansion of $\hor$ as $r \rightarrow 0^+$ start at order $Q = \rh + 2\rv$ and can be explicitly computed as
\begin{equation}
\hor = r^Q \ho^{(Q)} + r^{Q+1} \ho^{(Q+1)} + r^{Q+2} \ho^{(Q+2)} + O(r^{Q+3}), \text{ as } r \rightarrow 0^+.
\end{equation}

We calculate the homogeneous terms $\ho^{(j)}$ for $j=Q, Q+1, Q+2$. 

For $j=Q$, 
\begin{align}
\ho^{(Q)} &= \theta^{1(1)} \wedge \cdots \wedge \theta^{\rh(1)} \wedge \eta^{1(2)} \wedge \cdots \wedge \eta^{\rv(2)} \\
&= \frac{1}{2^\rv} dx^1 \wedge \cdots \wedge dx^\rh \wedge dz^1 \wedge \cdots \wedge dz^\rv.
\end{align}

It holds that $\ho^{(Q+1)} = 0$, since $\theta^{i(2)} =0$ and $\eta^{i(3)}$ is a linear combination of $dx^j$.  In fact, a similar argument will hold in general for $j$ odd. 

Finally, for $j= Q+2$ we have: 
\begin{equation}
\ho^{(Q+2)} = \frac{1}{6\cdot 2^\rv} R_{\gamma\alpha\beta}^\alpha (q) x^\beta x^\gamma dx^1 \wedge \cdots \wedge dx^\rh \wedge dz^1 \wedge \cdots \wedge dz^\rv.
\end{equation}
It follows that $a_{\rh,\rv}$ is given by 
\begin{equation}
a_{\rh,\rv} := \frac{1}{(4m)^{\rv/2}} \int_{\hat{B}(0,1)} dx^1 \cdots dx^\rh dz^1 \cdots dz^\rv
\end{equation}
and the coefficient of the $r^2$ term has the value
\begin{equation}
\frac{1}{6(4\rh)^{\rv/2}} \int_{\hat{B}(0,1)} R_{\gamma\alpha\beta}^\alpha (q) x^\beta x^\gamma dx^1 \cdots dx^\rh dz^1 \cdots dz^\rv.
\end{equation}

Since the ball ${{\hat{B}(0,1)}}$ is invariant under the reflection $x^\gamma \rightarrow -x^\gamma$ we conclude that: 
\begin{equation}
\int_{\hat{B}(0,1)}x^\beta x^\gamma \ dx^1 \cdots dx^\rh dz^1 \cdots dz^\rv
\end{equation}
is non-zero if and only if $\beta = \gamma$ and moreover the integral does not depend on this value.  Since $R_{\beta \alpha \beta}^\alpha = - R_{\beta \alpha \alpha}^\beta$ it follows that the coefficient in front of $r^2$ will be of the form $-b_{\rh,\rv}\kh(q)$ with
\begin{equation}
b_{\rh,\rv} := \frac{1}{6(4m)^{\rv/2}} \int_{\hat{B}(0,1)} (x^\alpha)^2 \ dx^1 \cdots dx^\rh dz^1 \cdots dz^\rv,
\end{equation}
where the index $\alpha$ can be chosen to be any of $1,2,\ldots,\rh$. 
\end{proof}

Being an equiregular sub-Riemannian manifold, $(\M,\Ho, g_\Ho)$ carries an intrinsic sub-Laplacian $\DsR$ induced by the Popp measure $\mathcal{P}$, (see  \cite{ABGR09,BR13}). More precisely, in terms of the special frame $\{X_1,\cdots,X_{\rh},Z_1,\cdots,Z_\rv\}$, this second order, positive and hypoelliptic differential operator $\DsR$ can be expressed explicitly in the following form (see \cite{BR13}):
\begin{equation}\label{sublaplacian}
\DsR=-\left(\sum_{\alpha=1}^{\rh}X_\alpha^2+\text{div}_{\mathcal{P}}(X_\alpha)X_\alpha\right).
\end{equation}
Here the divergence operator is defined by
\begin{equation}
\mathrm{div}_{\mathcal{P}}(X_\alpha) \mathcal{P}= \Lie_{X_\alpha}(\mathcal{P}).
\end{equation}

We now calculate the divergence in terms of the geometric data. Applying the Leibniz law for the Lie derivative,
\begin{equation}
\Lie_{X_\alpha}(\mathcal{P}) = \frac{1}{\rh^{\rv/2}} \sum_{a=1}^{\rh+\rv} \nu^1 \wedge \cdots \wedge \nu^{a-1} \wedge \Lie_{X_\alpha}(\nu^a) \wedge \nu^{a+1} \wedge \cdots \wedge \nu^{\rh+\rv}.
\end{equation}

Since $\{\nu^1,\cdots,\nu^{\rh+\rv}\}$ is a coframe, we write for $\alpha \in \{1, \cdots, \rh\}$ and $a \in \{1, \cdots, \rh+\rv\}$:
\begin{equation}
\Lie_{X_\alpha}(\nu^a) = f_{\alpha b}^{a}\nu^b.
\end{equation}
with the functions $f_{\alpha b}^a$ given by $f_{\alpha b}^a=\Lie_{X_\alpha}(\nu^a)(X_b)=\nu^a([X_b,X_\alpha]).$ Hence we obtain
\begin{equation}
\mathrm{div}_{\mathcal{P}}(X_\alpha) = \theta^\beta([X_\beta,X_\alpha]) + \eta^i([Z_i,X_\alpha])
\end{equation}

Now, using the properties of the Bott connection, a calculation shows that $\omega_\alpha := \theta^\beta([X_\beta,X_\alpha]) = \omega_{\alpha}^{\beta}(X_\beta)$ and $\eta^i([Z_i,X_\alpha])=0$; summarizing the above calculation gives us that
\begin{lem}\label{sublap}
The intrinsic sub-Laplacian $\DsR$ has the expression
\begin{equation}
\DsR=-\left(\sum_{\alpha=1}^{\rh}X_\alpha^2 + \omega_\alpha X_\alpha\right).
\end{equation}
\end{lem}

By definition the sub-Laplacian $\DsR$ is positive and it coincides up to a constant factor with the (negative) \emphs{horizontal Laplacian} defined in \cite[Remark 2.19]{BGRV21}.

%% file: sections/sec5_heat-invariants.tex
\section{Heat invariants}
\label{Section:Heat invariants}
%%%%%%%%%%%%%%%%%%%%%%%%%%%%%%%%%%%%%%%%%%%%%%%%%%%%%%%%%%%%%%%%%%%%%%%%%%%%%%%%%%
Using \cref{taylor_exp} and \cref{sublap}, we compute the low order homogeneous terms in the decomposition of the sub-Laplacian into a sum of 
homogeneous differential operators with polynomial coefficients:
$$\DsR=-\left(\hDsR+\mathcal{A}^{(-1)}+\mathcal{A}^{(0)}+\cdots\right) \hspace{2ex} \text{with} \hspace{2ex} 
\hDsR:=\sum_\alpha \hat{X}_\alpha^2.$$
%with $$\hat{\Delta}_{sub}=\sum_\alpha \hat{X}_\alpha^2$$
 Note that  $\hDsR$ is the intrinsic sub-Laplacian on the $H$-type group $\mathbb{G}(q)$ (cf. Remark \ref{Remark_3_5}). The operators $\mathcal{A}^{(-1)}$ and $\mathcal{A}^{(0)}$ are given by: 
\begin{align}
&\mathcal{A}^{(-1)}:=\sum_{\alpha} \Big(\hat{X}_\alpha X_\alpha^{(0)}+X_\alpha^{(0)}\hat{X}_\alpha\Big),
\\
&\mathcal{A}^{(0)}:=\sum_{\alpha} \Big(\hat{X}_\alpha X_\alpha^{(1)}+X_\alpha^{(1)}\hat{X}_\alpha+\left(X_\alpha^{(0)}\right)^2+\frac{1}{2}R_{\gamma\beta\alpha}^{\beta}(q)x^\gamma \hat{X}_\alpha\Big).
\end{align}
According to the results in \cite[p. 28,46]{VHT21}  and based on Duhamel's formula the second heat invariant $c_1(q)$ at $q$ is given by the Schwartz kernel 
$K_1(1,0,0)$ (at time $t=1$) of the operator
$$C_1(t):=\int_{0}^{t}e^{(t-s)\hDsR}\left(\mathcal{A}^{(0)}(q)e^{s\hDsR}+\mathcal{A}^{(-1)}(q)C_0(s)\right)ds,$$
where 
$$C_0(t)=\int_{0}^{t}e^{(t-s)\hDsR}\mathcal{A}^{(-1)}(q)e^{s\hDsR}ds.$$
From now on and in order to simplify the formulas we assume that the torsion is horizontally parallel, i.e. $\nabla_{\Ho}T=0$. 
According to \cref{properties}, (3) this implies that all horizontal derivatives of $J_{\alpha\beta}^{i}$ at the point $q$ vanish. As a consequence, at the point $q$, we have 
$J_{\alpha\beta}^{i(1)}=0$ by \cref{properties}, (1) and hence, $X_\alpha^{(0)}
%\frac{1}{3}J_{\alpha \beta}^{i(1)}x_{\beta}\hat{Z}_i
=0$ for all $\alpha$ by \cref{expansion}. With this in mind, it follows that $\mathcal{A}^{(-1)}=0$ under the condition $\nabla_{\Ho}T=0$, and 
the second heat invariant $c_1(q)$ is obtained from the Schwartz kernel $K_1(1,0,0)$ of 
$$C_1(t)=\int_{0}^{t}e^{(t-s)\hDsR}\mathcal{A}^{(0)}(q)e^{s\hDsR}ds$$
at $(0,0)$ and for time $t=1$. More precisely, we have
\begin{equation}\label{second_coef}
c_1(q) = \int_{0}^{1}\int_{\R^{\rh+\rv}} \hat{K}_2(s,0,\xi)(\mathcal{A}^{(0)}(q) \hat{K}_1)(1-s,\xi,0) d\xi ds,
\end{equation}
where $\hat{K}_1$ denotes the heat kernel associated with $\hDsR$  on $\mathbb{G}(q) \cong \R^{\rh+\rv}$ equipped with the nilpotentization of the Popp measure at the point $q$, and the operator $\mathcal{A}^{(0)}$ acts on the second component of $\hat{K}_1$.
\begin{lem}\label{data} 
 Assume that $\nabla_{\Ho}T=0$. Then the second heat invariant $c_1(q)$ is a linear combination of the following components of tensors at $q$: 
\begin{align}
& R_{\alpha\beta\gamma}^{\delta},\hspace{2mm}R_{\alpha\beta i}^{j},\hspace{2mm}J_{\alpha\lambda}^{i} R_{\alpha\beta\gamma}^{\delta},\hspace{2mm}J_{\delta\lambda}^{i} R_{\alpha\beta\gamma}^{\delta},\hspace{2mm}J_{\alpha\gamma}^{i} R_{\alpha\beta i}^{j},\\
& J_{\alpha\alpha^\prime}^{i}J_{\beta\beta^\prime}^{j}R_{\alpha\gamma\delta}^{\beta},\hspace{2mm}X_\alpha X_\beta(J_{\gamma\delta}^{i}),\hspace{2mm} J_{\alpha\alpha^\prime}^{i}X_\gamma X_\delta(J_{\alpha\beta}^{j}),\\
&Z_i(J_{\alpha\beta}^{j}),\hspace{2mm}J_{\alpha\beta}^{i}Z_j(J_{\alpha\gamma}^{k}).
\end{align}
\end{lem}
\begin{proof}
According to the data and notation in \cref{expansion}  and using $X_\alpha^{(0)}=0$ for all $\alpha$ we can express the operator $\mathcal{A}^{(0)}$ in the form 
$$\mathcal{A}^{(0)}=(f_{\alpha}^{\beta}+f_{\beta}^{\alpha})\hat{X}_\alpha\hat{X}_\beta
+\hat{X}_\alpha(f_\beta^\alpha)\hat{X}_\beta
+\Big(\hat{X}_\alpha(h^i_\alpha)\hat{Z}_i
+2h^{i}_{\alpha}\hat{X}_\alpha\hat{Z}_i
+\frac{1}{2}R_{\gamma\beta\alpha}^{\beta}(q)x^\gamma\hat{X}_\alpha\Big).$$
Then the statement directly follows from \cref{second_coef}. 
\end{proof}

By \cref{data}, we know that the second heat invariant $c_1(q)$ can be expressed in terms of components of certain tensors with respect to the orthonormal basis $\{X_\alpha,Z_i\}$ at $q$. In the following lemmas we will simplify these expressions. 

\begin{rem}
The key observation is that $c_1(q)$ is independent of the choice of such an orthonormal basis; it follows that the linear combinations occurring in \cref{data} must be invariant under the action of the group ${\bf O}(\rh) \times {\bf O}(\rv)$ and therefore we can use techniques from the classical invariance theory of the orthogonal group (cf. \cite{ABP73,W97}) to obtain a more precise expression of the second heat invariant in terms of components of the curvature and torsion tensors.
\end{rem}

\begin{lem}\label{Lem:traces}
The second heat invariant $c_1(q)$ is a linear combination of the following traces:

\begin{center}
	\begin{tabular}{ccc}
		$R_{\alpha\beta\beta}^{\alpha}$
		& $J_{\alpha\gamma}^{i}J_{\beta\delta}^{i} R_{\alpha\gamma\delta}^{\beta}$
 		& $J_{\alpha\gamma}^{i}J_{\beta\delta}^{i}R_{\alpha\delta\gamma}^{\beta}$ \\
 		$J_{\alpha\beta}^{i}X_\gamma X_\gamma(J_{\alpha\beta}^{i})$
 		& $J_{\alpha\beta}^{i}X_\beta X_\gamma(J_{\alpha\gamma}^{i})$
 		& $J_{\alpha\beta}^{i}X_\gamma X_\beta(J_{\alpha\gamma}^{i})$
	\end{tabular}
\end{center}

where we emphasize that we use the summation rule described in \cref{not:summations}.
\end{lem}

\begin{proof}
First we will write the coefficients of the tensors in their covariant form by contracting with the metric tensor. For instance
$
g_{\sigma\delta}R_{\alpha\beta\gamma}^{\delta}=R_{\alpha\beta\gamma\sigma}.
$
The tensor with the coefficients $R_{\alpha\beta\gamma\sigma}$ will be considered as an element of the vector space 
\begin{equation}
((\R^\rh)^*)^{\otimes 4} = (\R^\rh)^* \otimes (\R^\rh)^* \otimes (\R^\rh)^* \otimes (\R^\rh)^*.
\end{equation} 
More precisely, we denote $V = (\R^\rh)^*$ and write $\{e^1,\ldots,e^\rh\}$ for the standard dual basis in $V$. Then
\begin{equation}
\big{\{} R_{\alpha\beta\gamma\delta} \colon \alpha,\beta,\gamma,\delta \in \{1, \ldots, n\}\big{\}} \mapsto R_{\alpha\beta\gamma\delta} e^\alpha \otimes e^\beta \otimes e^\gamma \otimes e^\delta \in V^{\otimes 4}.
\end{equation}
The vector space $V$ can be considered as an ${\bf O}(\rh)$ module under the standard action of the orthogonal group. This action is extended to $V^{\otimes 4}$ componentwise. 

Having this example in mind, we proceed like in the CR case~\cite{BGS84}. We consider the space  $V = (\R^\rh)^*$ (resp. $W=(\R^\rv)^*$) with the standard dual basis $\{e^1,\cdots,e^\rh\}$ (resp. $\{f^1,\cdots,f^\rv\}$) as an ${\bf O}(\rh)$ \big(resp. ${\bf O}(\rv)$\big) module. These actions extend naturally to the tensor products
\begin{equation}
V^{\otimes r} \otimes W^{\otimes s} := \underbrace{V \otimes \cdots \otimes V}_{r \text{ times}} \otimes \underbrace{W \otimes \cdots \otimes W}_{s \text{ times}}.
\end{equation}
We write all the tensors occurring in \cref{data} in the  covariant form and consider them as elements of the corresponding vector space $E := V^{\otimes r} \otimes W^{\otimes s}$ for appropriate integers $r$ and $s$.

Note that this tensor representation is equivariant under the action of the group ${\bf O}(\rh)\times{\bf O}(\rv)$; that is, the ${\bf O}(\rh) \times {\bf O}(\rv)$ group action on the chosen orthonormal frame $\{X_1(q), \ldots, X_\rh(q), Z_1(q), \ldots, Z_\rv(q)\}$ at $q \in \M$ commutes with the representation map.  Let us denote by $\Theta \in E$ one of the tensors from \cref{data}, written in the covariant form in the fixed orthonormal frame at $q \in \M$. By the invariance theory \cite{ABP73,W97} there is a linear functional $f \colon E \rightarrow \R$, such that 
\begin{equation}
f(U\Theta)= c_1(q) \hspace{2ex} \text{\emphs{ for all }}\hspace{2ex} U\in {\bf{O}}(\rh)\times {\bf{O}}(\rv).
\end{equation}
By replacing the functional $f$ by its average over ${\bf{O}}(\rh)\times {\bf{O}}(\rv)$, we assume that $f$ is ${\bf{O}}(\rh)\times {\bf{O}}(\rv)$-invariant. It is known that an ${\bf{O}}(\rh)\times {\bf{O}}(\rv)$-invariant linear functional on $E = V^{\otimes r}\otimes W^{\otimes s}$ is non-zero only in the case where $r$ and $s$ are both even and moreover such a functional must be a complete contraction (up to a constant multiple). Hence, the second heat invariant $c_1(q)=f(\Theta)$ must be a linear combination of
\begin{center}
	\begin{tabular}{cccc}
		$R_{\alpha\beta\beta}^\alpha$
		& $R_{\alpha\alpha\beta}^\beta$
		& $R_{\alpha\beta\alpha}^\beta$
		& $J_{\alpha\gamma}^i J_{\beta\delta}^i R_{\alpha\gamma\delta}^\beta$ \\
		$J_{\alpha\gamma}^i J_{\beta\delta}^i R_{\alpha\delta\gamma}^\beta$ 
		& $J_{\alpha\beta}^i X_\gamma X_\gamma(J_{\alpha\beta}^i)$
		& $J_{\alpha\beta}^i X_\beta X_\gamma(J_{\alpha\gamma}^i)$
		& $J_{\alpha\beta}^i X_\gamma X_\beta(J_{\alpha\gamma}^i).$
	\end{tabular}
\end{center}

\cref{Lem:traces} follows now from the fact that the curvature tensor $R$ is skew-symmetric in the first two lower components so that $R_{\alpha\alpha\beta}^{\beta}=0$ and $R_{\alpha\beta\alpha}^{\beta}=-R_{\alpha\beta\beta}^{\alpha}$.  (Note that in this final equality summations are not necessary despite the convention \cref{not:summations}.)
\end{proof}

The following result can be found in \cite{BGRV21} and will be useful in the sequel:
\begin{lem}\label{Lemma_relations_connection_horizontally_parallel}
Let $(\M,\Ho,g_{\Ho})$ be an H-type foliation with horizontally parallel torsion, i.e. $\nabla_\Ho T=0$. Then (1) - (3) hold 
for $X\in  \s\Ho $ and $Z,W\in  \s\Vr $:
\begin{enumerate}
\item $(\nabla_Z J)_W=-(\nabla_W J)_Z$.
\item $(\nabla_Z J)_W$ is skew-symmetric and anti-commutes with $J_W$.
\item $\|(\nabla_Z J)_W X\|^2=\langle R(Z,W)W,Z\rangle\|X\|^2$.
\end{enumerate}
\end{lem} \label{H-type-foliations}

\begin{proof}
See Lemmas $2.6$, $2.18$ and $3.5$ in \cite{BGRV21}.
\end{proof}
To study the traces from \cref{Lem:traces} we introduce the following operator.
\begin{defn}
We define the bundle like operator $M(Z,W)\colon\Ho_q\to\Ho_q$  for $q \in \M$ by
\begin{equation}
X \mapsto M(Z,W)X = J_WJ_Z(\nabla_Z J)_WX
\end{equation}
for $X \in \s\Ho, Z,W \in \s\Vr$.
\end{defn}

\cref{Lemma_relations_connection_horizontally_parallel} immediately implies that $M(Z,W)=0$ for linearly dependent vectors $Z,W\in\Vr_q$, $q \in \M$. 

\begin{lem}\label{trace_index}
Under the assumptions of the preceding lemma, the following holds for $Z,W\in\Vr$:
\begin{enumerate}
\item The symmetric part of the operator $M(Z,W)$ is given by $$N(Z,W):=M(Z,W)+\langle Z,W\rangle(\nabla_Z J)_W.$$ In particular, $M(Z,W)$ is symmetric if $Z$ and $W$ are orthogonal  
$\langle Z,W\rangle=0$.
\item The eigenvalues of the operator $N(Z,W)$ are given by
$$\pm\sqrt{(\|Z\|^2\|W\|^2-\langle Z,W\rangle^2)\langle R(Z,W)W,Z\rangle}.$$
In particular, 
\begin{align}
\textup{tr}\left(M(Z,W)\right)&=\textup{tr}\left(N(Z,W)\right)\\
&=\sigma(Z,W)\sqrt{(\|Z\|^2\|W\|^2-\langle Z,W\rangle^2)\langle R(Z,W)W,Z\rangle},
\end{align}
where $\sigma(Z,W)$ denotes the number of positive eigenvalues minus the number of negative eigenvalues of $N(Z,W)$.
\item The operator $N(Z,W)$ is zero if and only if $\langle R(Z,W)W,Z\rangle=0$, i.e. the sectional curvature of the vertical plane generated by $Z$ and $W$ vanishes. 
Otherwise, $N(Z,W)$ is invertible.
\end{enumerate}
\end{lem}
 \begin{proof}
It is sufficient to prove the statements (1) and (2). 
\vspace{1mm}\\
(1): \cref{Lemma_relations_connection_horizontally_parallel} and the $H$-type condition \cref{H-type-condition-polarization} show: 
\begin{align}
N(Z,W)&=\frac{1}{2} \Big{(} M(Z,W)+M(W,Z) \Big{)}\\
&= \frac{1}{2} \Big{(} M(Z,W)+ J_Z J_W(\nabla_WJ)_Z \Big{)}\\
&=\frac{1}{2} \Big{(} M(Z,W)- ( -J_WJ_Z-2 \langle Z,W \rangle ) (\nabla_ZJ)_W \Big{)}\\
&=M(Z,W)+ \langle Z,W \rangle (\nabla_ZJ)_W. 
\end{align}
(2): Note that $M(Z,W)$ and $(\nabla_Z J)_W$ commute according to \cref{Lemma_relations_connection_horizontally_parallel}. Let $X$ be any 
eigenvector of $N(Z,W)$, then we obtain: 
\begin{align}\label{GL_1_Norm_N}
\| N(Z,W)X \|^2
&=\big{\langle} (M(Z,W)+ \langle Z,W \rangle (\nabla_ZJ)_W)^2X, X \big{\rangle}\\
&=\big{\langle} M(Z,W)^2X,X \big{\rangle} + \langle Z,W \rangle^2 \| (\nabla_ZJ)_W X \|^2. \notag
\end{align}
In the last equality we have used \cref{H-type-condition-polarization} and the skew-symmetry of $(\nabla_ZJ)_W$ which show that: 
\begin{align}
\big{\langle} M(Z,W) (\nabla_ZJ)_WX,X \big{\rangle}
&=- \big{\langle} J_WJ_Z (\nabla_ZJ)_WX, (\nabla_ZJ)_WX \big{\rangle}\\
&=  \langle Z,W \rangle \| (\nabla_ZJ)_WX \|^2. 
\end{align}
As for the first summand: 
\begin{align}
\big{\langle} M(Z,W)^2X,X \big{\rangle} 
&= - \big{\langle} J_WJ_Z(\nabla_Z J)_WX, J_ZJ_W(\nabla_ZJ)_WX \big{\rangle}\notag\\
&= \big{\|} J_W J_Z( \nabla_ZJ)_WX\big{\|}^2+2\langle Z,W \rangle \big{\langle} J_W J_Z (\nabla_ZJ)_W, (\nabla_ZJ)_WX \big{\rangle}\notag\\
&= \Big{(} \|Z\|^2\|W\|^2-2\langle Z,W \rangle^2 \Big{)} \big{\|} (\nabla_ZJ)_WX\big{\|}^2. \label{GL_2_Inner_product_M} 
\end{align}
Combining \cref{GL_1_Norm_N} and \cref{GL_2_Inner_product_M} with \cref{Lemma_relations_connection_horizontally_parallel}, this shows
\begin{equation}
\big{\|} N(Z,W)X\|^2= \Big{(} \|Z\|^2 \|W\|^2 - \langle Z,W \rangle^2 \Big{)} \big{\langle} R(Z,W)W,Z \big{\rangle} \|X\|^2,
\end{equation}
which completes the proof.
\end{proof}

We define an important invariant 
\begin{equation} \label{Definition_invariant_tau}
\tv:=\sum_{i,j=1}^{\rv}\tr \hspace{1mm}\big{(}M(Z_i,Z_j)\big{)},
\end{equation}
where $\{Z_1,\cdots,Z_\rv\}$ is a local orthonormal frame of the vertical distribution $\Vr$  and $\tr$ denotes the pointwise matrix trace of $M(Z_i,Z_j) \in \End(\Ho)$. Note that the definition of $\tv$ does not depend on the choice of an orthonormal frame of $\Vr$.

\begin{rem}\label{Remark_formula_for_tau}
If $Z_1,\cdots,Z_\rv$ is an orthonormal frame of $\Vr$, then by \cref{trace_index}, property $(2)$ above we can write
$$\tv =\sum_{ij}\sigma(Z_i,Z_j)\sqrt{R_{ijj}^{i}},$$
Since $\tv$ is independent of the choice of an orthonormal frame $Z_1,\cdots,Z_\rv$, it is natural to ask whether one can use this invariance to further simplify this expression. As we will see below, under further assumption on the leaves of our foliation the invariant $\tv$ will be everywhere constant.
\end{rem}
In the next proposition we need the notion of {\it horizontally parallel Clifford structure} on an $H$-type foliation. We start with a definition (see \cite[Definition 3.1]{BGRV21}). 
\begin{defn}
An $H$-type foliation $(\M, \mathcal{H}, g_{\mathcal{H}})$ with horizontally parallel torsion $T$ is said to have horizontally parallel Clifford structure if there exists a smooth 
bundle map $\Psi: \mathcal{V} \times \mathcal{V} \longrightarrow C\ell_2(\mathcal{V})$ (Clifford  bundle) such that for all $Z_1, Z_2 \in \Gamma(\mathcal{V})$: 
\begin{equation}\label{condition_horizontally_parallel_CS}
(\nabla_{Z_1}J)_{Z_2}= J_{\Psi(Z_1,Z_2)}.
\end{equation} 
\end{defn}
We refer to \cite{BGRV21} for examples and further properties. With this we can prove: |
\begin{prop}\label{Prop: eigenspaces}
Suppose $\rv \geq 2$ and assume that the torsion is horizontally parallel. Then
\begin{enumerate}
\item Let $Z,W \in \Vr$ be orthonormal, and suppose that $X \in \s\Ho$ is a $\lambda$-eigenvector for $M(Z,W) = N(Z,W) = J_WJ_Z (\nabla_Z J)_W$. The vectors $J_ZX, J_WX, J_ZJ_WX$ are all $\lambda$-eigenvectors for $N(Z,W)$.  
\item $\sigma = 4k$, $k\in\mathbb N$.  
\end{enumerate}
Assume moreover that the sectional curvature $\kv$ of the leaves is a positive constant; that is, for any $Z,W\in\Vr$
\begin{equation}
\langle R(Z,W)W,Z\rangle=\kv\left(\|Z\|^2\|W\|^2-\langle Z,W\rangle^2\right).
\end{equation}
It then follows that
\begin{enumerate}[resume]
\item $\sigma(Z,W)=\sigma$ is independent of the choice of  vectors $Z,W\in\Vr$ and it holds:
$$\tv=\rv(\rv-1)\sigma \sqrt{\kv}.$$ 
\item If $\M$ also has an horizontally parallel Clifford structure (see \cite[Theorem 3.6]{BGRV21}) then $\sigma$ and $\tv$ are constant; specifically
\begin{equation}
\sigma = -\rh \text{ and } \tv = -\rv(\rv-1)\rh\sqrt{\kv}.
\end{equation}
\end{enumerate}
\end{prop}

\begin{proof}  
(1) It follows from \cref{Lemma_relations_connection_horizontally_parallel} that $J_Z$ and $J_W$ anti-commute with $(\nabla_Z J)_W$, and 
from the $H$-type condition that they also anti-commute with $J_WJ_Z$; together these imply that $N(Z,W)$ commutes with each of 
$J_Z, J_W,$ and $J_ZJ_W$, completing the proof of (1). Lemma~\ref{trace_index} (2) shows that 
$$
\lambda=\pm\sqrt{\langle R(Z,W)W,Z\rangle}
$$
and the vector space ${\rm span}\{X,J_ZX, J_WX,J_ZJ_WX\}$ is an eigenspace with eigenvalue $\lambda$. At the point $q\in \M$ at the level of the tangent group it will form a subgroup isomorphic to the complexified Heisenberg group.

(2) $N(Z,W)$ is symmetric, therefore diagonalizable. Partition $\Ho_q$ with $q \in \M$ into eigenspaces spanned by $X_1,(J_ZX)_q, (J_WX)_q,(J_VJ_WX)_q$, which is always possible by (1); this immediately implies $\sigma = 4k$ for some $k \in \mathbb Z$.  	

(3) Let $Z,W\in\Vr$ be linearly independent. By \cref{trace_index}, we know that
\begin{equation}
\tr(N(Z,W)) = \sigma(Z,W) \sqrt{(\|Z\|^2\|W\|^2 - \langle Z,W\rangle^2)\langle R(Z,W)W, Z \rangle}.
\end{equation}
The idea is to consider  $\sigma$ as a function defined on the Grassmann $2$-plane bundle $G_2(\Vr)$ as follows:
\begin{equation}
\sigma(Z,W)=\frac{\textup{tr}(N(Z,W))}{\sqrt{(\|Z\|^2\|W\|^2 - \langle Z,W\rangle^2)\langle R(Z,W)W,Z\rangle}}.
\end{equation}
Note that the right hand side of this equation depends only on the $2$-plane generated by $Z$ and $W$. In that way, we consider $\sigma \colon \G_2(\Vr) \rightarrow \R$ which is a smooth function with discrete values in $\mathbb{Z}$. Since $\G_2(\Vr)$ is connected, it follows that $\sigma$ is a constant function.

Hence 
\begin{equation}
\tv = \sum_{i\neq j,\ i,j=1}^{\rv} \tr(M(Z_i,Z_j)) = \sum_{i\neq j,\ i,j=1}^{\rv} \sigma \sqrt{\kv} = \rv(\rv-1)\sigma\sqrt{\kv}
\end{equation}
by $\tr(M(Z,W)) = \tr(N(Z,W))$. This shows that $\tv$ is constant on $\M$ with value 
\begin{equation}
\tv=\rv(\rv-1)\sigma \sqrt{\kv}.
\end{equation}

(4) According to \cite[Theorem 3.6]{BGRV21} the assumption of having horizontally parallel Clifford structure implies that the sectional curvature $\kv$ of the leaves of the foliation associated to $\Vr$ is constant and 
\begin{equation}
(\nabla_{Z_i}J)_{Z_j} = -\sqrt{\kv} J_{Z_i} J_{Z_j}
\end{equation}
Then all eigenvalues are the same $\lambda = -\sqrt{\kv}$.  The claim follows immediately. 
\end{proof}

\begin{prop}
Assume that the torsion is horizontally parallel and recall from \eqref{Def: scalar curvature} the scalar curvature $\kh$. Then at $q\in \M$:	
\begin{enumerate}
\item $J_{\alpha\gamma}^{i}J_{\beta\delta}^{i}R_{\alpha\delta\gamma}^{\beta}=\langle R(X_\alpha,J_{Z_i}X_\beta)J_{Z_i}X_\alpha,X_\beta\rangle=\kh+ 2\tv,$
\item $J_{\alpha\gamma}^{i}J_{\beta\delta}^{i}R_{\alpha\gamma\delta}^{\beta}=\langle R(X_\alpha,J_{Z_i}X_\alpha)J_{Z_i}X_\beta,X_\beta\rangle=2\kh + 4\tv, $
\item $J_{\alpha\beta}^{i}X_\gamma X_\gamma(J_{\alpha\beta}^{i}) = 0, $
\item $J_{\alpha\beta}^{i}X_\beta X_\gamma(J_{\alpha\gamma}^{i})=\frac{1}{2}\tv,$
\item $J_{\alpha\beta}^{i}X_\gamma X_\beta(J_{\alpha\gamma}^{i})=-\frac{1}{2}\tv.$
\end{enumerate}
We recall that the summations are implied, per \ref{not:summations}.
\end{prop}

\begin{proof}
(1) We recall the following formula for the commutator of $R(\cdot,\cdot)$ and $J$ on horizontal vectors.  A proof of \cref{Commutator_R_J} can be found in \cite[Lemma 3.18]{BGRV21}. 
\begin{equation}\label{Commutator_R_J}
[R(X,Y), J_Z] = \left(\nabla_{T(X,Y)} J \right)_Z + J_{ (\nabla_Z T) (X,Y)}
\end{equation}
for $X,Y\in \s\Ho$ and $Z \in \s\Vr$. Hence, we can write
\begin{align}
J_{\alpha\gamma}^i J_{\beta\delta}^i R_{\alpha\delta\gamma}^\beta &= \langle R(X_\alpha,J_{Z_i} X_\beta) J_{Z_i} X_\alpha, X_\beta\rangle \\
&= \langle J_{Z_i} R(X_\alpha,J_{Z_i}X_\beta) X_\alpha, X_\beta \rangle \\
&+\left\langle \left(\nabla_{T(X_\alpha, J_{Z_i}X_\beta)}J \right)_{Z_i} X_\alpha, X_\beta \right\rangle + \left\langle J_{ \left(\nabla_{Z_i} T) (X_\alpha, J_{Z_i}X_\beta) \right)} X_\alpha, X_\beta \right\rangle.
\end{align}

Since $J_{Z_i}$ is an isometry, writing $J_{Z_i}X_\beta = X_\gamma$ and $X_\beta = -J_{Z_i} X_\gamma$ the first term on the right hand side $\langle J_{Z_i} R(X_\alpha,J_{Z_i}X_\beta) X_\alpha, X_\beta \rangle$ can be interpreted as the scalar curvature $\kh$. The second term can be rewritten in the form
\begin{align}
\left\langle \left(\nabla_{T(X_\alpha,J_{Z_i}X_\beta)} J \right)_{Z_i} X_\alpha,X_\beta \right\rangle
&= -\left\langle \left( \nabla_{T(X_\alpha,X_\beta)}J \right)_{Z_i} X_\alpha, J_{Z_i}X_\beta \right\rangle \\
&= -\langle J_{Z_j}X_\alpha, X_\beta \rangle \cdot \left\langle \left(\nabla_{Z_j}J \right)_{Z_i} X_\alpha, J_{Z_i}X_\beta \right\rangle \\
&= \left\langle J_{Z_j}X_\alpha, J_{Z_i}\left(\nabla_{Z_j}J \right)_{Z_i} X_\alpha \right\rangle \\
&= \left\langle X_\alpha, J_{Z_j}J_{Z_i} \left(\nabla_{Z_i}J \right)_{Z_j} X_\alpha \right\rangle \\
&= \tr\left(J_{Z_j}J_{Z_i} \left(\nabla_{Z_i}J \right)_{Z_j}\right).
\end{align}

Finally, the last term can be expressed as
\begin{align}
\Big{\langle} J_{ \left(\nabla_{Z_i} T)(X_\alpha,J_{Z_i}X_\beta)\right)} X_\alpha,X_\beta \Big{\rangle}&
=-\Big{\langle} J_{ \left(\nabla_{Z_i} T)(X_\alpha,X_\beta)\right)} X_\alpha,J_{Z_i}X_\beta\Big{\rangle}\\
&=-\langle  \left(\nabla_{Z_i} T\right)(X_\alpha,X_\beta),Z_j\rangle\cdot \langle J_{Z_j}X_\alpha,J_{Z_i}X_\beta\rangle\\
&=\langle\left(\nabla_{Z_i}J\right)_{Z_j}X_\alpha,X_\beta\rangle\cdot\langle J_{Z_i}J_{Z_j}X_\alpha,X_\beta\rangle\\
&=\langle J_{Z_j}J_{Z_i}\left(\nabla_{Z_i}J\right)_{Z_j}X_\alpha,X_\alpha\rangle\\
&=\text{tr}\left(J_{Z_j}J_{Z_i}\left(\nabla_{Z_i}J\right)_{Z_j}\right).
\end{align}
Here in the second line we used that $(\nabla_{Z_i} T)(X_\alpha,X_\beta)$ has components only in the direction of $\Vr$.  For the third line we applied the formula
\begin{equation}
\langle \left(\nabla_WJ\right)_ZX,Y\rangle=\langle Z,\left(\nabla_W T\right)(X,Y)\rangle,\quad X,Y\in\s\Ho ,\quad Z,W\in\s\Vr 
\end{equation}
that is the consequence of the differentiation
\begin{equation}
W\left(\langle J_Z X,Y\rangle\right)=W\left(\langle Z,T(X,Y)\rangle\right).
\end{equation}

(2) Using the first Bianchi identity \cref{First_Bianchi_identity} and the fact that $T(\Ho,\Ho)\subset \Vr$, $T(\Ho,\Vr)=0$, we obtain: 
\begin{multline}
\langle R(X_\alpha,J_{Z_i}X_\alpha)J_{Z_i}X_\beta,X_\beta\rangle+\\ +\langle R(J_{Z_i}X_\alpha,J_{Z_i}X_\beta)X_\alpha,X_\beta\rangle+\langle R(J_{Z_i}X_\beta,X_\alpha)J_{Z_i}X_\alpha,X_\beta\rangle=0.
\end{multline}
By using $J_{Z_i}X_\alpha=X_{\gamma}$, $X_{\alpha}=-J_{Z_i}X_\gamma$ for the second term and skew symmetry of the curvature tensor with respect to the two first vectors, see \cite[Lemma 3.7]{H12}, we obtain the desired result.

(3) By the H-type condition $J_{Z_i}^2=-\Id_\Ho$, it follows that $J_{\alpha\beta}^{i} J_{\alpha\beta}^{i}=1$ for all $\alpha$. Taking the first derivative along $X_\gamma$, we obtain $J_{\alpha\beta}^{i} X_\gamma\left(J_{\alpha\beta}^{i}\right)=0$ for all $\alpha, \beta,i$, and $\gamma$. Taking again the derivative along $X_\gamma$, we obtain
\begin{equation}
X_\gamma\left(J_{\alpha\beta}^{i}\right)X_\gamma\left(J_{\alpha\beta}^{i}\right)+J_{\alpha\beta}^{i}X_\gamma X_\gamma\left(J_{\alpha\beta}^{i}\right)=0.
\end{equation}
Since the torsion is horizontally parallel  and $T(\Ho,\Ho)\subset \Vr$ it follows that $X_\gamma\left(J_{\alpha\beta}^{i}\right)(q)$ vanishes at $q\in \M$ and hence
\begin{equation}
J_{\alpha\beta}^{i}X_\gamma X_\gamma\left(J_{\alpha\beta}^{i}\right)=0.
\end{equation}

(4) Using the H-type condition $J_{Z_i}^2=-\Id_\Ho$, it follows that $J_{\alpha\beta}^{i} J_{\alpha\gamma}^{i}=\delta_{\beta\gamma}$ for all $\beta$ and $\gamma$. Hence, using the assumption that the torsion is horizontally parallel and \cref{properties} (3), we can write at $q \in \M$
\begin{equation}
J_{\alpha\beta}^{i}X_\beta X_\gamma\left(J_{\alpha\gamma}^{i}\right)+X_\beta X_\gamma\left(J_{\alpha\beta}^{i}\right)J_{\alpha\gamma}^{i}=0.
\end{equation}
This shows that at $q$
\begin{equation}
J_{\alpha\beta}^{i}X_\beta X_\gamma\left(J_{\alpha\gamma}^{i}\right)=-J_{\alpha\beta}^{i}X_\gamma X_\beta\left(J_{\alpha\gamma}^{i}\right).
\end{equation}
Now, we write at $q$
\begin{align}
J_{\alpha\beta}^{i}X_\beta X_\gamma\left(J_{\alpha\gamma}^{i}\right)&=\frac{1}{2} J_{\alpha\beta}^{i} [X_\beta,X_\gamma]\left(J_{\alpha\gamma}^{i}\right)\\
&=\frac{1}{2} J_{\alpha\beta}^{i} J_{\beta\gamma}^{j}Z_j\left(J_{\alpha\gamma}^{i}\right)\\
&=-\frac{1}{2} \langle J_{Z_i}X_\alpha,J_{Z_j}X_\gamma\rangle\langle \left(\nabla_{Z_j}J\right)_{Z_i}X_\alpha,X_\gamma\rangle\\
&=\frac{1}{2}\tv(q).
\end{align}

(5) The last statement can be shown by similar calculation.
\end{proof}

\begin{rem}
Note that the traces in $(1)$ and $(2)$ from the preceding proposition were computed in \cite{IMV14} for quaternionic contact manifolds endowed with the Biquard connection. Therein it was shown that these traces are of the form $C\kh$ where $C$ is a dimensional constant and $\kh$ is the scalar curvature with respect to the Biquard connection. Furthermore, any $3$-Sasakian manifold can be considered both as an H-type foliation (endowed with the Bott connection) with horizontally parallel torsion, or as a quaternionic contact manifold (endowed with the Biquard connection). In that setting, our result about the traces in $(1)$ and $(2)$ coincides with the result from \cite{IMV14}, because both local invariants $\kh$ and $\tv$ are global constants and related to each other by some dimensional constant.
\end{rem}

Now, we summarize the preceding computations in the following main theorem:
\begin{thm}\label{thmA}
Let $(\M,\Ho,g_\Ho)$ be an H-type foliation with horizontally parallel torsion. Then the second heat invariant $c_1$ is a linear combination of the local invariants $\kh$ and $\tv$:
\begin{equation}
c_1 =C_1\kh + C_2\tv,
\end{equation}
where $C_1$ and $C_2$ are universal constants depending only on $\rh$ and $\rv$.
\end{thm}

%% file: sections/sec6_open-problems.tex
\section{Open Problems}
%%%%%%%%%%%%%%%%%%%%%%%%%%%%%%%%%%%%%%%%%%%%%%%%%%%%%%%%%%%%%%%%%%%%%%%%%%%
We present a list of open problems that seem closely related to our analysis: 

Let us consider the open subset $\mathcal{U}$ of the Grassmann 2-plane bundle $\G_2(\Vr)$ defined by
\begin{equation}
\mathcal{U}:= \{Z \wedge W \in \G_2(\Vr): S(Z \wedge W) \neq 0\},
\end{equation}
where $S(Z \wedge W)$ denotes the sectional curvature of the plane spanned by $Z,W$ and the map $\sigma: \mathcal{U}\longrightarrow \mathbb{Z}$, which assigns to $Z \wedge W$ the difference of the numbers of positive and negative eigenvalues of $N(Z,W)$. 

Note that $\mathcal{U} = \emptyset$ if and only if the torsion $T$ is completely parallel. Assume now that $\mathcal{U}\neq\emptyset$, then it is natural to ask: 
\begin{enumerate}
\item[(a)] {\it Is the map $\sigma$ constant on $\mathcal{U}$?}
\end{enumerate}
Let $\{Z_1, \ldots , Z_m\}$ be an orthonormal frame of $\Vr$. Then, the invariant $\tau$ can be expressed as (see \cref{Remark_formula_for_tau} and \cref{trace_index}, (2)):
\begin{equation}
\tau = \sum_{i,j} \sigma_{ij} \sqrt{R_{ijj}^i},
\end{equation}
where $\sigma_{ij}:=\sigma(Z_i \wedge Z_j).$
\begin{enumerate}
\item [(b)] {\it Are the $\sigma_{ij}$'s independent of $i$ and $j$? Are the sectional curvatures $R_{ijj}^{i}$ independent of $i$ and $j$? }
\end{enumerate}
Note that $\tau$ is independent of the choice of an orthonormal basis of the vertical space, which suggests that the answer may be ``yes''. 

%% file: sections/sec7_appendix.tex
\section{Appendix}
%\appendix{A}
In this appendix we present detailed proofs of technical lemmas which were applied in the previous sections. 
\begin{prop} \label{Prop_1_Appendix}
With the notation in \cref{Section:Priviledged Coordinates} and \cref{H-type-foliations} we have the identities:
\begin{itemize}
\item[(a)]  $d\theta^\alpha=\theta^\beta\wedge\omega_{\beta}^{\alpha}$
\item[(b)]  $d\eta^i=\frac{1}{2}J^{i}_{\beta\gamma}\theta^\beta\wedge\theta^\gamma+\eta^k\wedge\omega_{k}^{i}$
\item[(c)]  $d\omega_{a}^{b}=\frac{1}{2}R_{\beta\gamma a}^{b}\theta^\beta\wedge\theta^\gamma+R_{\beta j a}^{b}\theta^\beta\wedge\eta^j+\frac{1}{2}R_{jka}^{b}\eta^j\wedge\eta^k+\omega_{a}^{c}\wedge\omega_{c}^{b}$. 
\end{itemize}
\end{prop} 
\begin{proof}
(a):  Let $\alpha \in \{1, \ldots, n\}$ and consider an expansion: 
\begin{align}\label{Expansion_d_theta}
d\theta^{\alpha}
&=\sum_{\beta < \gamma} a_{\beta \gamma}^{\alpha} \theta^{\beta} \wedge \theta^{\gamma} + \sum_{\beta, i}b_{\beta i}^{\alpha} \theta^{\beta} \wedge \eta^i + \sum_{i <j}c_{ij}^{\alpha} \eta^i \wedge \eta^j. 
\end{align}
Since $\theta^{\beta}$ is dual to $X_{\beta}$  and the torsion of horizontal vectors is vertical we conclude that the coefficients $a_{\beta \gamma}^{\alpha} $ are given by: 
\begin{align}
a_{\beta \gamma}^{\alpha} 
&=2d \theta^{\alpha}(X_{\beta}, X_{\gamma}) =\underbrace{X_{\beta}\theta^{\alpha}(X_{\gamma})}_{=0} -
\underbrace{X_{\gamma}\theta^{\alpha}(X_{\beta})}_{=0} 
-\theta^{\alpha} \big{(} [X_{\beta}, X_{\gamma}] \big{)}\\
&=-  g \big{(} X_{\alpha}, \big{[}X_{\beta}, X_{\gamma} \big{]} \big{)}\\
&=- g\big{(} X_{\alpha}, \nabla_{X_{\beta}}X_{\gamma} \big{)}+ g \big{(} X_{\alpha}, \nabla_{X_{\gamma}}X_{\beta}\big{)}+\underbrace{g\big{(}X_{\alpha}, T(X_{\beta}, X_{\gamma}) \big{)}}_{=0}. 
\end{align}
Recall that $\omega_{\alpha}^{c}$ denote the connection one-forms, i.e.  $\nabla X_{\alpha}= X_{c} \otimes \omega_{\alpha}^{c}$ and therefore: 
\begin{align}
a_{\beta \gamma}^{\alpha} 
&= -g\big{(} X_{\alpha}, X_{c} \otimes \omega_{\gamma}^{c}(X_{\beta}) \big{)}+g\big{(} X_{\alpha}, X_{c} \otimes \omega_{\beta}^{c} (X_{\gamma}) \big{)}=- \omega_{\gamma}^{\alpha}(X_{\beta})+ \omega_{\beta}^{\alpha}(X_{\gamma}). 
\end{align}
%Using the expansion $\omega_{\alpha \beta}= \omega_{\alpha \beta}(X_c)\theta^{c}$ 
Hence, we find: 
\begin{align}
\sum_{\beta < \gamma} a_{\beta \gamma}^{\alpha} \theta^{\beta} \wedge \theta^{\gamma} 
&= - \sum_{\beta <\gamma}\omega_{\gamma}^{\alpha} (X_{\beta}) \theta^{\beta} \wedge \theta^{\gamma}+ 
\sum_{\beta <\gamma} \omega_{\beta}^{\alpha} (X_{\gamma})\theta^{\beta} \wedge \theta^{\gamma} \label{GL_first_term}\\
&= \omega_{\beta}^{\alpha} (X_{\gamma}) \theta^{\beta} \wedge \theta^{\gamma}.\notag
\end{align}
Similarly, we calculate $b_{\beta i}^{\alpha}=-g ( X_{\alpha}, [X_{\beta}, Z_i])$.  Since $T(X_{\beta}, Z_i) =0$ according to property (3) of the connection we find: 
\begin{align}
b_{\beta i}^{\alpha} &=- g\big{(} X_{\alpha}, \underbrace{\nabla_{X_{\beta}}Z_i }_{\in \Gamma(\Vr)}\big{)}+ g \big{(} X_{\alpha}, \nabla_{Z_i}X_{\beta}\big{)}
=g\big{(} X_{\alpha}, \nabla_{Z_i}X_{\beta}\big{)}=\omega_{\beta}^{\alpha}(Z_i).  
\end{align}
Therefore: 
\begin{equation} \label{GL_Second_term}
\sum_{\beta, i} b_{\beta i}^{\alpha} \theta^{\beta} \wedge \eta^i= \sum_{\beta, i} \omega_{\beta}^{\alpha}(Z_i) \theta^{\beta} \wedge \eta^i.  
\end{equation}
Finally, since $T(Z_i,Z_j)=0$ according to (3): 
\begin{align}
c_{ij}^{\alpha}&=2d\theta^{\alpha}(Z_i,Z_j)=- g \big{(} X_{\alpha}, \big{[}Z_i, Z_j \big{]} \big{)}
=- g \big{(} X_{\alpha}, \underbrace{\nabla_{Z_i}Z_j}_{\in \Gamma(\Vr)}- \underbrace{\nabla_{Z_j}Z_i}_{\in \Gamma(\Vr)} \big{)}=0. \label{GL_Third_term}
\end{align}
Combining \cref{GL_first_term}, \cref{GL_Second_term} and \cref{GL_Third_term} together with $\omega_{\alpha}^{\beta}= \omega_{\alpha}^{\beta}(X_c)\theta^{c}$  yields (a): 
\begin{equation}
d\theta^{\alpha}= %\sum_{\gamma=1}^m 
\theta^{\beta} \wedge \big{[} \omega_{\beta}^{\alpha}(X_{\gamma})\theta^{\gamma} + \omega_{\beta}^{\alpha} (Z_i) \eta^i \big{]}=\theta^{\beta} \wedge \omega_{\beta}^{\alpha}. 
\end{equation} 
(b): Let $i \in \{1, \ldots, m\}$ and consider an expansion of the left hand side: 
\begin{align}
d\eta^i
&=\sum_{\beta < \gamma}  d_{\beta \gamma}^i \theta^{\beta} \wedge \theta^{\gamma} + \sum_{\beta, j} e_{\beta j}^i \theta^{\beta} \wedge \eta^j+ \sum_{j <k}f_{jk}^i \eta^j \wedge \eta^k. 
\end{align}
We calculate the coefficients $ d_{\beta \gamma}^i$, $e_{\beta j}^i $ and $f_{jk}^i$. First, note that: 
\begin{align}
d_{\beta \gamma}^i&= 2d \eta^i(X_{\beta}, X_{\gamma})=-  \eta^i\big{(} [X_{\beta}, X_{\gamma}] \big{)}=- g\big{(}Z_i, [X_{\beta}, X_{\gamma}]\big{)}\\
&=  -g \big{(} Z_i, \underbrace{\nabla_{X_{\beta}}X_{\gamma}}_{\in \Gamma(\Ho)}- \underbrace{\nabla_{X_{\gamma}} X_{\beta}}_{\in \Gamma(\Ho)} - T(X_{\beta}, X_{\gamma}) \big{)}\\
&=  g\big{(} Z_i, T(X_{\beta}, X_{\gamma}) \big{)}\\
&= g\big{(} J_{Z_i}X_{\beta}, X_{\gamma} \big{)}=J_{\beta \gamma}^i. 
\end{align}
Similarly, for $e_{\beta j}^i$ and $f_{jk}^i$:
\begin{align}
e_{\beta j}^i
%&=2d \eta^i\big{(}X_{\beta}, Z_j \big{)}=-\eta^i\big{(} [X_{\beta}, Z_j] \big{)}=- g\big{(} Z_i, [X_{\beta}, Z_j] \big{)}\\
&=-  g \big{(} Z_i, \underbrace{\nabla_{X_{\beta}}Z_j}_{\in \Gamma(\Vr)}- \underbrace{\nabla_{Z_j} X_{\beta}}_{\in \Gamma(\Ho)} - 
\underbrace{T(X_{\beta}, Z_j)}_{=0} \big{)}\\
&=- g \big{(} Z_i, \nabla_{X_{\beta}}Z_j \big{)}= -\omega_{j}^{i}(X_{\beta}),  \\
%\end{align}
%Therefore, 
%\begin{equation}
%\sum_{\beta, j} e_{\beta j}^i \theta^{\beta} \wedge \eta^j= - \sum_{\beta, j}\omega_{ji} (X_{\beta}) \theta^{\beta} \wedge \eta^j. 
%=- \frac{1}{2} \omega_{ji} \wedge \eta^j= {\red \frac{1}{2}} \eta^j \wedge \omega_{ji}.% \hspace{3ex} \mbox{\red (Check the factor $\frac{1}{2}$!)}
%\end{equation}
%Now, we calculate $f_{jk}^i$. As before: 
%\begin{align}
f_{jk}^i%&= - \eta^i \big{(} [Z_j, Z_k] \big{)} =- g\big{(} Z_i, \big{[} Z_j, Z_k \big{]} \big{)}\\
&=- g \big{(} Z_i, \underbrace{\nabla_{Z_j}Z_k}_{\in \Gamma(\Vr)}- \underbrace{\nabla_{Z_k} Z_j}_{\in \Gamma(\Vr)} - 
\underbrace{T(Z_j, Z_k)}_{=0} \big{)}\\
&=- \omega_{k}^{i}(Z_j) + \omega_{j}^{i}(Z_k). 
%- \frac{1}{2} g\big{(}Z_i, \nabla_{Z_j}Z_k \big{)}=-\omega_{ki}(Z_j). 
\end{align}
Combining the identities and using the skew-symmetry $J_{\beta \gamma}^i = - J_{\gamma \beta}^i$ gives: 
\begin{align}
d \eta^i
&= \sum_{\beta < \gamma} J_{\beta \gamma}^i \theta^{\beta} \wedge \theta^{\gamma} - \sum_{\beta, j} \omega_{j}^{i}(X_{\beta}) \theta^{\beta} \wedge \eta^j
- \sum_{j<k} \Big{(} \omega_{k}^{i}(Z_j) - \omega_{j}^{i}(Z_k) \Big{)} \eta^j \wedge \eta^k\\
&= \frac{1}{2} J_{\beta \gamma}^i \theta^{\beta} \wedge \theta^{\gamma}+ \sum_{\beta, k} \omega_{k}^{i}(X_{\beta}) \eta^k \wedge \theta^{\beta}
+ \sum_{j,k}\omega_{k}^{i}(Z_j) \eta^k \wedge \eta^j\\
&=\frac{1}{2} J_{\beta \gamma}^i \theta^{\beta} \wedge \theta^{\gamma}+\sum_k \eta^k \wedge \left( \sum_{\beta} \omega_{k}^{i}(X_{\beta})\theta^{\beta}+ \sum_j \omega_{k}^{i}(Z_j) \eta^j \right)\\
&=\frac{1}{2} J_{\beta \gamma}^i \theta^{\beta} \wedge \theta^{\gamma}+ \eta^k \wedge \omega_{k}^{i}. 
\end{align}
%Therefore: 
%\begin{equation}
%f_{jk}^i \eta^j \wedge \eta^k={\red - \frac{1}{2} \omega_{ki}(Z_j) \eta^j \wedge \eta^k =}
%\end{equation}
%\todo{Paste everything together}
\noindent
(c):   Let $a,b \in \{1, \ldots, n+m\}$ and consider an expansion of the left hand side: 
\begin{align}\label{Expansion_GL_3}
d\omega_{a}^{b}
&= \sum_{\beta < \gamma} n_{a\beta \gamma}^{b} \theta^{\beta} \wedge \theta^{\gamma} +\sum_{\beta, j} h_{a\beta j}^{b}\theta^{\beta} \wedge \eta^j
+ \sum_{j<k} m_{ajk}^{b}\eta^j \wedge \eta^k. 
% \hspace{3ex} \mbox{where} \hspace{3ex} 
%(\beta < \gamma, j<k). 
\end{align}
Again, we calculate the coefficients in the expansion. 
\begin{align}
n_{a\beta \gamma}^{b}
&=2\,d\omega_{a}^{b}(X_{\beta}, X_{\gamma})= X_{\beta} \omega_{a}^{b}(X_{\gamma})-  X_{\gamma} \omega_{a}^{b}(X_{\beta}) - \omega_{a}^{b}
\big{(} [X_{\beta}, X_{\gamma}] \big{)}. 
\end{align}
Recall that  $\nabla_{X_{\gamma}}X_{a}= X_{c} \otimes \omega_{a}^{c}(X_{\gamma})$ and therefore $\omega_{a}^{b}(X_{\gamma})= g\big{(}X_b, \nabla_{X_{\gamma}} X_{a} \big{)}$. 
According to (1) the connection $\nabla$ is metric and hence: 
\begin{align}
 X_{\beta} \omega_{a}^{b}(X_{\gamma})
 &=X_{\beta}g\big{(}X_b, \nabla_{X_{\gamma}} X_{a} \big{)}
 =g\big{(} \nabla_{X_{\beta}}X_b, \nabla_{X_{\gamma}} X_{a} \big{)}+g \big{(} X_b, \nabla_{X_{\beta}} \nabla_{X_{\gamma}}X_a \big{)}\\
X_{\gamma} \omega_{a}^{b}(X_{\beta}) 
&=g\big{(} \nabla_{X_{\gamma}}X_b, \nabla_{X_{\beta}} X_{a} \big{)}+g \big{(} X_b, \nabla_{X_{\gamma}} \nabla_{X_{\beta}}X_a \big{)}. 
\end{align}
Moreover, one has: 
\begin{align}
\omega_{a}^{b} \big{(} [X_{\beta}, X_{\gamma} \big{]} \big{)} &= g\big{(} X_b, \nabla_{[X_{\beta}, X_{\gamma}]}X_a \big{)}\\
g \big{(} \nabla_{X_{\beta}}X_b, \nabla_{X_{\gamma}}X_a \big{)}
&
=g \big{(} X_{\rho} ,\nabla_{X_{\beta}}X_b\big{)} g\big{(}  X_{\rho},\nabla_{X_{\gamma}}X_a  \big{)}
= -\omega^{b}_{\rho}(X_{\beta}) \omega_{a}^{\rho}(X_{\gamma}),
\end{align}
where we used the property that for all $a,b,c$: 
\begin{equation}
\omega_{a}^{c}(X_b)= g(X_c, \nabla_{X_b}X_a)= -g(\nabla_{X_b}X_c, X_a)= -\omega_{c}^{a}(X_b), 
\end{equation}
based on the fact that the Bott connection $\nabla$ is metric and that the basis is orthonormal. It also shows the skew-symmetry of the connection form; that is $\omega_{a}^{c}=-\omega_{c}^{a}$. 

Inserting these identities into the equations above yields: 
\begin{align}
n_{a\beta \gamma}^{b}&
=g \big{(} X_b, R(X_{\beta}, X_{\gamma})X_a \big{)} 
+ \omega_{a}^{\rho}(X_{\beta}) \omega^{b}_{\rho}(X_{\gamma})-  \omega_{a}^{\rho}(X_{\gamma}) \omega^{b}_{\rho}(X_{\beta})
\\
&
=R_{\beta \gamma a}^b
+ \omega_{a}^{\rho}(X_{\beta}) \omega^{b}_{\rho}(X_{\gamma})-  \omega_{a}^{\rho}(X_{\gamma}) \omega^{b}_{\rho}(X_{\beta}).
\end{align}
A similar calculation shows% of $h_{\beta j}^{ab}$:
\begin{align}
h_{a\beta j}^{b}
&=
2\,d\omega_{a}^{b}(X_{\beta}, Z_j)= X_{\beta} \big{(} \omega_{a}^{b}(Z_j) \big{)}- Z_j \big{(} \omega_{a}^{b}(X_{\beta}) \big{)} - \omega_{a}^{b}
\big{(} [X_{\beta}, Z_j] \big{)}\\
&=
R_{\beta j a}^b
+ \omega_{a}^{\rho}(X_{\beta}) \omega^{b}_{\rho}(Z_j)-\omega_{a}^{\rho}(Z_j)\omega^{b}_{\rho}(X_{\beta}), 
\end{align}
and 
\begin{align}
m_{ajk}^{b}
&=2\,d \omega_{a}^{b}(Z_j,Z_k))=  Z_j \omega_{a}^{b}(Z_k)- Z_k\omega_{a}^{b}(Z_j)- \omega_{a}^{b} \big{(}[Z_j,Z_k] \big{)} \\
%&= g(X_b, R(Z_j,Z_k)X_a)+\omega_{b\rho}(Z_j) \omega_{a\rho}(Z_k)-\omega_{b\rho}(Z_k) \omega_{a\rho}(Z_j)\\
&=R_{jka}^b
+\omega_{a}^{\rho}(Z_j) \omega^{b}_{\rho}(Z_k)-\omega_{a}^{\rho}(Z_k) \omega^{b}_{\rho}(Z_j). 
\end{align}
Observe that for $a,b,c,d \in \{1, \ldots, n+m\}$,
\begin{equation}
\omega_{a}^{c} \wedge \omega^{b}_{c}= \omega_{a}^{c}(X_d)\omega^{b}_{c}(X_l)\theta^d\wedge \theta^l. 
\end{equation}
Using the identities in the expansion \cref{Expansion_GL_3} together with the skew-symmetry of the coefficients $n_{a\beta\gamma}^{b}=- n_{a\gamma \beta}^{b}$ and $m_{ajk}^{b}=- m_{akj}^{b}$ defined above implies: 
\begin{align}
d\omega_{a}^{b}=
&
\Big{(}\frac{1}{2} R_{\beta \gamma a}^b 
+ \omega_{a}^{\rho}(X_{\beta}) \omega^{b}_{\rho}(X_{\gamma}) \Big{)} \theta^{\beta} \wedge \theta^{\gamma} \\
&
+ R_{\beta j a}^b \theta^{\beta} \wedge \eta^j
+ \omega_{a}^{\rho}(X_{\beta}) \omega^{b}_{\rho}(Z_j) \theta^{\beta} \wedge \eta^j
- \omega_{a}^{\rho}(Z_j) \omega^{b}_{\rho}(X_{\beta}) \theta^{\beta} \wedge \eta^j
\\
&
+ \Big{(} \frac{1}{2} R_{jka}^b 
+ \omega_{a}^{\rho}(Z_j) \omega^{b}_{\rho}(Z_k) \Big{)} \eta^j \wedge \eta^k 
\\
&
=\frac{1}{2} R_{\beta \gamma a}^b  \theta^{\beta} \wedge \theta^{\gamma}
+ R_{\beta j a}^b \theta^{\beta} \wedge \eta^j
+\frac{1}{2} R_{jka}^b \eta^j \wedge \eta^k 
+ \omega_{a}^{c} \wedge 
\omega^{b}_{c}. 
\end{align}
We have used $\omega_{b}^{c}(X_{\beta})=0$ if $\beta \in \{1, \ldots, n\}$ and $c >n$. 
\end{proof}